\documentclass[11pt,a4paper]{amsart}

\usepackage{amssymb,amsmath}
\usepackage[dvips]{color}
\usepackage[dvips,final]{graphics}
\usepackage{amscd}

\usepackage{pdfsync}
\usepackage{amsbsy}

\usepackage{color}

\theoremstyle{plain}
\newtheorem{theorem}{Theorem}[section]
\newtheorem{proposition}[theorem]{Proposition}
\newtheorem{corollary}[theorem]{Corollary}

\newtheorem{lemma}[theorem]{Lemma}
\theoremstyle{definition}
\newtheorem{definition}[theorem]{Definition}

\newtheorem{remark}[theorem]{Remark}

\makeindex                                    
\usepackage{pdfsync}

\newcommand{\rn}[1]{{\mathbb R}^{#1}}
\newcommand{\R}{\mathbb R}

\newcommand{\he}[1]{{\mathbb H}^{#1}}

\newcommand{\covH}[1]{{\bigwedge\nolimits^{#1}{\mfrak h}}}
\newcommand{\vetH}[1]{{\bigwedge\nolimits_{#1}{\mfrak h}}}

\newcommand{\covh}[1]{{\bigwedge\nolimits^{#1}{\mfrak h_1}}}
\newcommand{\veth}[1]{{\bigwedge\nolimits_{#1}{\mfrak h_1}}}

\newcommand{\scal}[2]{\langle {#1} , {#2}\rangle}

\newcommand{\Scal}[2]{\langle {#1} \vert {#2}\rangle}
\newcommand{\scalp}[3]{\langle {#1} , {#2}\rangle_{#3}}

\newcommand{\ccheck}{{\vphantom i}^{\mathrm v}\!\,}

\newcommand{\mc}{\mathcal }

\newcommand{\mfrak}{\mathfrak}

\newcommand{\N}{\mathbb N}

\newcommand{\curl}{\mathrm{curl}\;}

\newcommand{\Hom}{\mathrm{Hom}\,}

\title[Gagliardo-Nirenberg inequalities in Heisenberg groups
] 
{Gagliardo-Nirenberg inequalities for differential forms in Heisenberg groups}

\author[Annalisa Baldi, Bruno Franchi, Pierre Pansu]{
Annalisa Baldi\\
Bruno Franchi\\ Pierre Pansu
}

\thanks{ \\
 A.B. and B.F. are supported by University of Bologna, funds for selected research 
topics, by GNAMPA of INdAM, Italy and by MAnET Marie Curie Initial Training Network. 
\\  P.P. is supported by Agence Nationale de la Recherche, ANR-10-BLAN 0116.
}

\begin{document}

%


\begin{abstract} The $L^1$-Sobolev inequality  states that for compactly supported functions $u$ on the Euclidean $n$-space, 
the $L^{n/(n-1)}$-norm of a compactly supported function
 is controlled by the $L^1$-norm of its gradient.
The generalization to differential forms (due to Lanzani \& Stein and Bourgain \& Brezis) is recent, and states that a the $L^{n/(n-1)}$-norm of a compactly supported differential
$h$-form is controlled by the $L^1$-norm of its exterior differential $du$ and its exterior codifferential $\delta u$ (in special cases the 
$L^1$-norm must be replaced by the $\mc H^1$-Hardy norm).  We shall extend this result to Heisenberg groups in the framework of an appropriate complex of differential forms.

\end{abstract}

\maketitle

%
%

%
%
\section{Introduction}\label{heisenberg}

The $L^1$-Sobolev inequality (also known as Gagliardo-Nirenberg inequality) states that for compactly supported functions $u$ on the Euclidean $n$-space, 
\begin{equation}\label{GNineq}
\|u\|_{L^{n/(n-1)}(\R^n)}\le c\|\nabla u\|_{L^1(\R^n)}.
\end{equation}
The generalization to differential forms is recent (due to  Bourgain \& Brezis and Lanzani \& Stein), and states that the $L^{n/(n-1)}$-norm of a compactly supported differential
$h$-form is controlled by the $L^1$-norm of its exterior differential $du$ and its exterior codifferential $\delta u$ (in special cases the 
$L^1$-norm must be replaced by the $\mc H^1$-Hardy norm).  We shall extend this result to Heisenberg groups in the framework of an appropriate complex of differential forms.

\subsection{The Euclidean theory}

In a series of papers (\cite{BB2003}, \cite{BB2004}, \cite{BB2007}), Bourgain and Brezis establish new estimates for the Laplacian, 
the div-curl system, and more general Hodge systems in $\rn n$ and
they show in particular that
if $\stackrel{\rightarrow} F$ is a compactly supported smooth vector field
  in $\R^n$, with $n\ge 3$, and if $\curl \stackrel{\rightarrow} F = \stackrel{\rightarrow} f$
and $\mathrm{div}\, \stackrel{\rightarrow} F = 0$, then
there exists a constant $C>0$ so that
\begin{equation}\label{BB04}
\|\stackrel{\rightarrow} F\|_{L^{n/(n-1)}(\R^n)}\le\|  \stackrel{\rightarrow}  f\|_{L^1(\R^n)}\,.
\end{equation}
This result does not follow straightforwardly from Calder\`on-Zygmund theory and Sobolev inequality.
Indeed, suppose for sake of simplicity $n=3$ and let
$\stackrel{\rightarrow} F$ be a compactly supported smooth vector field,  and consider the system
\begin{equation}\label{system intro}
\left\{\begin{aligned}
\curl \stackrel{\rightarrow} F= \stackrel{\rightarrow} f \\ \mathrm{div}\,\stackrel{\rightarrow} F = 0\,. &
\end{aligned}
\right.
\end{equation}
 It is well known that $\stackrel{\rightarrow}F =(-\Delta)^{-1}\curl \stackrel{\rightarrow} f$ is a solution of \eqref{system intro}. 
 Then, by Calder\'on-Zygmund theory we can say that
$$
\|\nabla \stackrel{\rightarrow} F\|_{L^{p}(\R^3)}\le C_p\|\stackrel{\rightarrow} f\|_{L^p(\R^3)}\,, \quad \mathrm{for}\quad {1<p<\infty}.
$$
Then, by  Sobolev inequality, if  $1<p<3$ we have:
$$
\|  \stackrel{\rightarrow} F\|_{L^{p*}(\R^3)}\le\| \stackrel{\rightarrow}f\|_{L^p(\R^3)}\,,
$$ where $\frac{1}{p*}=\frac{1}{p}-\frac{1}{3}$. When we turn to the case $p=1$ the first inequality fails. The second remains true. 
This is exactly the result proved  by Bourgain and Brezis.


In \cite{LS} Lanzani \& Stein proved that \eqref{GNineq}
is the first link of a chain of analogous inequalities for compactly supported
smooth differential $h$-forms in $\rn n$, $n\ge 3$,
\begin{eqnarray}
\label{0a} &\|u\|_{L^{n/(n-1)}(\rn n)}\le C\, \big( \|du \|_{L^1(\rn n)} + \|\delta u \|_{L^1(\rn n)}\big) & \mbox{if $h\neq 1, n-1$;}
\\
\label{0b} &
\|u\|_{L^{n/(n-1)}(\rn n)}\le C\, \big( \|du \|_{L^1(\rn n)} + \|\delta u \|_{\mc H^1(\rn n)}\big) &\mbox{if $h=1$;}
\\ \label{0c} &
\|u\|_{L^{n/(n-1)}(\rn n)}\le C\, \big( \|du \|_{\mc H^1(\rn n)} + \|\delta u \|_{L^1(\rn n)}\big) &\mbox{if $h=n-1$,}
\end{eqnarray}
where $d$ is the exterior differential, and $\delta$ (the exterior
codifferential) is its formal $L^2$-adjoint. Here $\mc H^1(\rn n)$ is the real Hardy space (see e.g. \cite{stein}). In other words, the main
result of \cite{LS} provides a priori estimates for a div-curl systems with data in $L^1(\rn n)$. 
We stress that inequalities \eqref{0b} and \eqref{0c}  fail if we replace the Hardy norm with
the $L^1$-norm. Indeed (for instance), the inequality 
\begin{equation}\label{0bbis} 
\|u\|_{L^{n/(n-1)}(\rn n)}\le C\, \big( \|du \|_{L^1(\rn n)} + \|\delta u \|_{ L^1(\rn n)}\big) 
\end{equation}
is false for $1$-forms.
The counterexample is given by E.M. Stein in \cite{stein_diff}, p. 191.
Indeed, take $f_k\in \mc D(\mathbb R^n)$  such that $\|f\|_{L^1(\he n)} = 1$ for all $k\in \N$
and such that $(f_k)_{k\in\N}$ tends to the Dirac $\delta$ in the sense of distribution. Set now $v_k:=
\Delta^{-1}f_k$. Then estimate \eqref{0bbis} would yield that $\{ |\nabla v_k|  \; ; \;k\in N\}$ is bounded in
$L^{n/(n-1)}(\mathbb R^n)$, and then, taking the limit as $k\to\infty$ that $|x|^{-n} \in L^1(\mathbb R^n)$.

\subsection{The Heisenberg setting}

Recently, in \cite{CvS2009}, Chanillo \& Van Schaftingen extended Bourgain-Brezis inequality to a class of vector fields
in Carnot groups. Some of the results of \cite{CvS2009} are presented in Theorems \ref{chanillo_van}
and \ref{chanillo_van5} below in the setting of Heisenberg groups. These are the main tool that allows us to give a Heisenberg version of Lanzani \& Stein's result. We describe now the operators that will enter our theorem.

We denote by  $\he n$  the $n$-dimensional Heisenberg
group. It is well known that the Lie algebra $\mathfrak h$ of the left-invariant vector fields
admits the stratification $\mathfrak h = \mathfrak h_1\oplus \mathfrak h_2$. We shall refer to the elements
of $\mathfrak h_1$ as to the {\sl horizontal derivatives} on $\he n$.

Heisenberg groups admit a one-parameter group of automorphisms called dilations. Whereas, in Euclidean space, all  exterior forms of degree $h$ have homogeneity $h$ under Euclidean dilations, on the contrary, because of the stratification of $\mathfrak h$,
$h$-forms on Heisenberg groups split into two weight spaces, with weights $h$ and $h+1$. This leads to a modification $(E_0^*,d_c)$ of the de Rham complex introduced by Rumin. Bundles of covectors are replaced by subbundles $E_0^h$ and the exterior differentials by differential operators $d_c$ on spaces $\Gamma(E_0^h)$ of smooth sections of these subbundles.

It turns out that this complex, which is both invariant under left-translations and dilations, is easier to work with that ordinary differential forms.

The core of Rumin's theory relies on the following result.

\begin{theorem} \label{dc in H} If $0\le h\le 2n+1$  there exists a linear map
$$
d_c: \Gamma(E_0^h) \to \Gamma(E_0^{h+1})
$$
such that
\begin{itemize}
\item[i)] $d_c^2=0$ (i.e. $E_0:=(E_0^*,d_c)$ is a complex);
\item[ii)] the complex $E_0$ is exact;
\item[iii)] $d_c: \Gamma(E_0^h)\to \Gamma(E_0^{h+1})$ is an homogeneous differential operator in the 
horizontal derivatives
of order 1 if $h\neq n$, whereas $d_c: \Gamma(E_0^n)\to \Gamma(E_0^{n+1})$ is an homogeneous differential operator in the 
horizontal derivatives
of order 2;
\item[iv)] if $0\le h\le n$, then $\ast E_0^h = E_0^{2n+1-h}$;
\item[v)] the operator $\delta_c := (-1)^{h(2n+1)}\ast d_c \ast $ is the formal $L^2$-adjoint of $d_c$.
\end{itemize}

\end{theorem}

\begin{definition} If $0\le h\le 2n+1$, $1\le p \le \infty$,
we denote by $L^p( \he n, E_0^h)$
the space of all sections of $E_0^h$ such that their
components with respect a given left-invariant basis belong to
$L^p(\he n)$, endowed with its natural norm. Clearly, this definition
is independent of  the choice of the basis itself.
If $h=0$, we write $L^p(\he n)$ for $L^p(\he n, E_0^0)$.

The notations $\mc D(\he n, E_0^h)$, $\mc S(\he n, E_0^h)$, $\mc E(\he n, E_0^h)$, and
$\mc H^1(\he{n}, E_0^h)$ have an analogous  meaning (here  $\mc H^1$ is the Hardy space in $\he n$ defined in 
\cite{folland_stein}, p.75).
\end{definition}

Now can state our main result that generalizes the results of \cite{BF7} to all Heisenberg groups.

\begin{theorem}\label{Hn} Denote by $(E_0^*,d_c)$ the Rumin's complex in $\he n$, $n>2$
(for the cases $n=1,2$ we refer to \cite{BF7}).
Then there exists $C>0$ such that for any $h$-form
$u\in \mc D(\he n, E_0^h)$, $0\le h\le 2n+1$, such that
$$
\left\{\begin{aligned}
 d_c u = f \\ \delta_c u = g&
\end{aligned}
\right.
$$
we have:
\begin{itemize}
\item[i)] if $h=0,2n+1$, then
 \begin{align*} &\|u\|_{L^{Q/(Q-1)}(\he{n})\hphantom{, E_0^{2n+1}}}\le   C\|f \|_{L^1(\he{n}, E_0^{1})} ;
\\
&\|u\|_{L^{Q/(Q-1)}(\he{n}, E_0^{2n+1})}\le    C\|g \|_{L^1(\he {n},E_0^{2n})} ;
\end{align*}

\item[ii)] if $h=1,2n$, then
 \begin{align*} 
 &\|u\|_{L^{Q/(Q-1)}(\he{n}, E_0^1)}\le  C\big( \|f \|_{L^1(\he{n},E_0^2)} + \|g \|_{\mc H^1(\he {n})}\big);
\\
&\|u\|_{L^{Q/(Q-1)}(\he{n}, E_0^{2n})}\le   C\big( \|f \|_{\mc H^1(\he{n}, E_0^{2n+1})} + \|g \|_{L^1(\he {n}, E_0^{2n-1})}\big) ;
\end{align*}

\item[iii)]  if  $1<h<2n$ and $h\neq n,n+1$, then
 \begin{align*} 
&\|u\|_{L^{Q/(Q-1)}(\he{n}, E_0^h)}\le   C\big( \|f \|_{L^1(\he{n}, E_0^{h+1})} + \|g \|_{L^1(\he{n}, E_0^{h-1})}\big) ;
\end{align*}

\item[iv)] if $h=n,n+1$, then
 \begin{align*} 
&\|u\|_{L^{Q/(Q-2)}(\he{n}, E_0^n)}\le   C\big( \|f \|_{L^1(\he{n}, E_0^{n+1})} + \|d_cg \|_{L^1(\he{n}, E_0^{n})}\big) ;
\\
&\|u\|_{L^{Q/(Q-2)}(\he{n}, E_0^{n+1})}\le   C\big(\|\delta_cf \|_{L^1(\he{n}, E_0^{n+1})} + \|g \|_{L^1(\he{n}, E_0^{n})}  \big) ;
\\
& \|u\|_{L^{Q/(Q-1)}(\he{n}, E_0^n)}\le C \|g \|_{L^1(\he{n}, E_0^{n-1})} \quad\mbox{if $f=0$;}
\\
& \|u\|_{L^{Q/(Q-1)}(\he{n}, E_0^{n+1})}\le C\|f \|_{L^1(\he{n}, E_0^{n+2})} \quad\mbox{if $g=0$.}
\end{align*}
\end{itemize}

\end{theorem} 

The proof of Theorem \ref{Hn} follows  the lines of the proofs in \cite{BF7} of the corresponding results for
$\he 1$ and $\he 2$. The new crucial contribution of the present paper in contained in Theorem \ref{pierre} that,
roughly speaking, states that the components 
with respect to a given basis of closed forms in $E_0^h$
are linear combinations of the components of a horizontal vector field  with vanishing  ``generalized horizontal divergence''.
This is obtained by proving that the symbol of the intrinsic differential $d_c$ is left-invariant and invertible (see
Corollary \ref{irreducible} and Proposition \ref{invariance}).

In Section \ref{notations} we fix the notations we shall use throughout this paper. In Section \ref{heisenberg tensor}
we gather some more or less known results about tensor analysis in Heisenberg groups. Section \ref{af} recalls results borrowed to Chanillo and Van Schaftingen. Section \ref{main} contains Theorem \ref{pierre} together with several auxiliary results.  
The proof of Theorem \ref{Hn} is completed in Section \ref{analysis}.
Section \ref{final} contains a variant of Theorem \ref{Hn} where no differential operator occurs on the right hand side.

\section{Notations and definitions} \label{notations}

As above, we denote by $\he n$ the $n$-dimensional Heisenberg group
 identified  with $\rn {2n+1}$ through exponential
coordinates. A point $p\in \he n$ is denoted by
$p=(x,y,t)$, with both $x,y\in\rn{n}$
and $t\in\R$.
   If $p$ and
$p'\in \he n$,   the group operation is defined as
\begin{equation*}
p\cdot p'=(x+x', y+y', t+t' + \frac12 \sum_{j=1}^n(x_j y_{j}'- y_{j} x_{j}')).
\end{equation*}

In particular for any $p\in\he n$ there is a familiy of (left) translations $\tau_p:\he n\to\he n$ defined  
by 
$$
\tau_p q := p\cdot q, \quad q\in \he n.
$$

For a general review on Heisenberg groups and their properties, we
refer to \cite{stein}, \cite {GromovCC}, \cite{BLU}, and to \cite {VarSalCou}.
We limit ourselves to fix some notations, following \cite{FSSC_advances}.

    We denote by  $\mfrak h$
 the Lie algebra of the left
invariant vector fields of $\he n$. As customary, $\mathfrak h$ is identified
with the tangent space $T_e\he n$ at the origin. The standard basis of $\mfrak
h$ is given, for $i=1,\dots,n$,  by
\begin{equation*}
X_i := \partial_{x_i}-\frac12 y_i \partial_{t},\quad Y_i :=
\partial_{y_i}+\frac12 x_i \partial_{t},\quad T :=
\partial_{t}.
\end{equation*}
The only non-trivial commutation  relations are $
[X_{j},Y_{j}] = T $, for $j=1,\dots,n.$ 
The {\it horizontal subspace}  $\mfrak h_1$ is the subspace of
$\mfrak h$ spanned by $X_1,\dots,X_n$ and $Y_1,\dots,Y_n$.
Coherently, from now on, we refer to $X_1,\dots,X_n,Y_1,\dots,Y_n$
(identified with first order differential operators) as to
the {\it horizontal derivatives}. Denoting  by $\mfrak h_2$ the linear span of $T$, the $2$-step
stratification of $\mfrak h$ is expressed by
\begin{equation*}
\mfrak h=\mfrak h_1\oplus \mfrak h_2.
\end{equation*}

\bigskip

The stratification of the Lie algebra $\mfrak h$ induces a family of non-isotropic dilations
$\delta_\lambda$, $\lambda>0$ in $\he n$. The homogeneous dimension  of $\he n$
with respect to $\delta_\lambda$, $\lambda>0$ is
$$
Q= 2n+2.
$$

The vector space $ \mfrak h$  can be
endowed with an inner product, indicated by
$\scalp{\cdot}{\cdot}{} $,  making
    $X_1,\dots,X_n$,  $Y_1,\dots,Y_n$ and $ T$ orthonormal.
    
Throughout this paper,
to avoid cumbersome notations, we write also
\begin{equation}\label{campi W}
W_i:=X_i, \quad W_{i+n}:= Y_i, \quad W_{2n+1}:= T, \quad \text
{for }i =1, \cdots, n.
\end{equation}

The dual space of $\mfrak h$ is denoted by $\covH 1$.  The  basis of
$\covH 1$,  dual to  the basis $\{X_1,\dots , Y_n,T\}$,  is the family of
covectors $\{dx_1,\dots, dx_{n},dy_1,\dots, dy_n,\theta\}$ where 
$$ \theta
:= dt - \frac12 \sum_{j=1}^n (x_jdy_j-y_jdx_j)$$ is called the {\it contact
form} in $\he n$. 

We indicate as $\scalp{\cdot}{\cdot}{} $ also the
inner product in $\covH 1$  that makes $(dx_1,\dots, dy_{n},\theta  )$ 
an orthonormal basis. The same notation will be used to denote the
scalar product in $\vetH 1$  that makes $(X_1,\dots, X_{n}, T )$ 
an orthonormal basis. 

Coherently with the previous notation \eqref{campi W},
we set
\begin{equation*}
\omega_i:=dx_i, \quad \omega_{i+n}:= dy_i, \quad \omega_{2n+1}:= \theta, \quad \text
{for }i =1, \cdots, n.
\end{equation*}

We put
$       \vetH 0 := \covH 0 =\R $
and, for $1\leq k \leq 2n+1$,
\begin{equation*}
\begin{split}
 \vetH k& :=\mathrm {span}\{ W_{i_1}\wedge\dots \wedge W_{i_k}: 1\leq
i_1< \dots< i_k\leq 2n+1\}= :\mathrm {span}\,\Psi_k,   \\
         \covH k& :=\mathrm {span}\{ \omega_{i_1}\wedge\dots \wedge \omega_{i_k}:
1\leq i_1< \dots< i_k\leq 2n+1\}= :\mathrm {span}\,\Psi^k.
\end{split}
\end{equation*}

The inner product $\scalp{\cdot}{\cdot}{} $ extends canonically to
$\vetH k $ and to $\covH k$ making both bases $\Psi_k$ and $\Psi^k$
orthonormal. 

If $1\leq k\leq 2n+1$, we denote by $\ast$ the Hodge isomorphism
\begin{equation*}
\ast : 
 \covH k \longleftrightarrow \covH{2n+1-k}
\end{equation*}
associated with the scalar product $\scalp{\cdot}{\cdot}{} $ and the volume form
$$
dV:= \omega_1\wedge\cdots\wedge \omega_{2n}\wedge\theta.
$$

The same construction can be performed starting from the vector
subspace $\mfrak h_1\subset \mfrak h$,
obtaining the {\it horizontal $k$-vectors} and {\it horizontal
$k$-covectors}
\begin{equation*}
\begin{split}
         \veth k& :=\mathrm {span}\{ W_{i_1}\wedge\dots \wedge W_{i_k}:
1\leq
i_1< \dots< i_k\leq 2n\}   \\
         \covh k& :=\mathrm {span}\{ \omega_{i_1}\wedge\dots \wedge \omega_{i_k}:
1\leq i_1< \dots< i_k\leq 2n\}.
\end{split}
\end{equation*}

It is well known that the Lie algebra $\mathfrak h$ can be identified
with the tangent space at the origin $e=0$ of $\he n$, and hence
 the horizontal layer   $\mathfrak h_1$
can
be identified with a subspace  of $T\he{n}_e$
that we can still denote by $\veth 1$.

In addition, the  symplectic 2-form   
$$-d \theta =
\sum _{i=1}^n dx_i \wedge dy_{i}$$
induces on $\mathfrak h_1$  a symplectic structure.  We point out that this
symplectic structure is compatible with our fixed scalar product $\scal{\cdot}{\cdot}$
and with the canonical almost complex structure on $\mathfrak h_1\equiv \mathbb C^n$.

Horizontal $k$-vectors can be identified with skew-symmetric $k$-tensor in
$\otimes^k \mathfrak h_1$. 

To fix our notations, we remind the following definition.
\begin{definition}\label{pullback} If $V,W$ are finite dimensional linear vector spaces and
$S:V\to W$ is a linear isomorphism, we define a map
$$\otimes_{r} S :
\otimes^r V\to \otimes^r  W 
$$ 
as the linear map defined by
$$
(\otimes_{r} S) (v_1\otimes\cdots\otimes v_r)=
S(v_1)\otimes\cdots\otimes S(v_r),
$$
and a map
$$\otimes^{s} S :
\otimes^s W^*\to \otimes^s V^*$$
as the linear map defined by
$$
\Scal {(\otimes^s S)(\alpha)}{v_1\otimes\cdots\otimes v_s}=
\Scal{\alpha}{(\otimes_s S)(v_1\otimes\cdots\otimes v_s)}
$$
for any $\alpha\in \otimes^s W^*$
 and any  $s$-tensor 
$v_1\otimes\cdots\otimes v_s\in
\otimes^s V$.
Finally, we define
$$
(\otimes^s_rS ) : (\otimes^r V )\otimes (\otimes^s V^* )\to ( \otimes^r W) \otimes (\otimes^s W^*)
$$
as follows:
$$
(\otimes^s_r S )(v\otimes w) :=  (\otimes_r S)(v)\otimes (\otimes^s S^{-1} )(w).
$$
\end{definition}

Throughout this paper, we shall deal with $(r,s)$-tensors  in
$$
\big(\otimes^r\mathfrak h\big)  \otimes \big( \otimes^s\mathfrak h^*\big),
$$
with $r,s\in \mathbb Z$, $r,s\ge 0$, that, in turn define a left-invariant fiber bundle
over $\he n$, that we still denote by $\big(\otimes^r\mathfrak h\big)  \otimes \big( \otimes^s\mathfrak h^*\big)$
as follows: first we identify $\big(\otimes^r\mathfrak h\big)  \otimes \big( \otimes^s\mathfrak h^*\big)$
with a subspace of
$\big(\otimes^r T_e\he n \big)  \otimes \big( \otimes^s T_e^*\he n\big)$ that we denote by
$$\big(\otimes^r\mathfrak h\big)_e  \otimes \big( \otimes^s\mathfrak h^*\big)_e .$$
Then the fiber of $\big(\otimes^r\mathfrak h\big)_p  \otimes \big( \otimes^s\mathfrak h^*\big)_p$ of
$\big(\otimes^r\mathfrak h\big)  \otimes \big( \otimes^s\mathfrak h^*\big)$
over $p\in \he n$ is
$$
\big(\otimes^r\mathfrak h\big)_p  \otimes \big( \otimes^s\mathfrak h^*\big)_p:= 
(\otimes^h d\tau_{p}(e))\big(\otimes^r\mathfrak h\big)_e
 \otimes 
 (\otimes^h d\tau_{p^{-1}}(p))\big( \otimes^s\mathfrak h^*\big)_e.
$$

The elements of the space of smooth sections of this bundle, i.e.
$$
\Gamma\big(\he n, \big(\otimes^r\mathfrak h\big)  \otimes \big( \otimes^s\mathfrak h^*\big)\big),
$$
are called $(r,s)$-tensors fields.

A special instance will be the {\sl horizontal} tensors belonging to 
$$
\big(\otimes^r\mathfrak h_1\big)  \otimes \big( \otimes^s\mathfrak h_1^*\big).
$$
The horizontal $(r,0)$-tensors fields will be called also horizontal $r$-vector fields.
The skew-symmetric horizontal $(0,s)$-tensors fields  are identified with the
horizontal differential forms.

Moreover, to avoid cumbersome notations, from now on, when dealing with a vector bundle $\mc N$
over $\he n$, if there is no way to misunderstand
we shall write also
$$
\Gamma(\mc N)\quad\mbox{for} \quad \Gamma(\he n, \mc N).
$$

Finally, a subbundle $\mc N$ of $ \big(\otimes^r\mathfrak h\big)  \otimes \big( \otimes^s\mathfrak h^*\big)$ is
said {\it left invariant} if
\begin{equation}\label{linvariant}
\mc N_p = (\otimes_r^s d\tau_p(e) )\mc N_e.
\end{equation}

It is customary (see e.g. \cite{weil}, Chapter I) to denote by
$$
L: \covh h\to \covh {h+2}
$$
the Lefschetz operator defined by $L\alpha := d\theta\wedge\alpha$, and by $\Lambda$ its
dual operator with respect to $\scal{\cdot}{\cdot}$. If $2\le h\le 2n$, we denote by $P^h\subset \covh{h}$
the space of primitive $h$-covectors defined by 
\begin{equation}\label{al1}
P^1:=\covh 1 \quad\mbox{and}\quad P^h:= \ker\Lambda\cap \covh{h}, \quad 2\le h\le 2n.
\end{equation}

Following \cite{rumin_jdg}, \cite{rumin_gafa}, for $h=0,1,\dots,2n+1$ we define a linear subspace $E_0^h$,  of 
 $\covH{h}$   as follows:

\begin{definition} \label{rumin in H} We set
\begin{itemize}
\item if $1\le h\le n$ then $E_0^h= P^h $;
\item if $n< h\le 2n+1$ then $E_0^h = \{\alpha = \beta\wedge\theta, \; \beta\in \covh{h-1},
\; L\beta = 0\}.$
\end{itemize} 
\end{definition}

\begin{remark} Definition \ref{rumin in H} is not the original definition due to M. Rumin,
but is indeed a characterization of Rumin's classes that is proved in \cite{rumin_jdg}
(see also \cite{pansu_bourbaki}, \cite{BF7} and \cite{BBF}). 
\end{remark}

By \eqref{linvariant}, the spaces $E_0^*$  define a family of left-invariant
subbundles (still denoted by $E_0^h$,
$h=0,\dots,2n+1$). It turns out that we  can identify
$E_0^h$ and $(E_0^h)_e$.

\section{Basic facts on tensor analysis in Heisenberg groups}\label{heisenberg tensor}
The following decomposition theorem
holds.

\begin{proposition}\label{sym deco} The space of 2-contravariant horizontal tensors $\otimes^2\mathfrak h_1$
can be written as a direct (orthogonal) sum
$$
\otimes^2\mathfrak h_1 = \mathrm{Sym}\, (\otimes^2\mathfrak h_1)
\oplus
\veth{2}
$$
of the space $\mathrm{Sym}\, (\otimes^2\mathfrak h_1)$ of the symmetric 2-tensors
and of the space $\veth{2}$ of the skew-symmetric 2-tensors.

An orthonormal basis of $\mathrm{Sym}\, (\otimes^2\mathfrak h_1)$ with
respect to the canonical scalar product in $\otimes^2\mathfrak h_1 $ is
$$
\left\{ \dfrac12 (W_i\otimes W_j + W_j\otimes W_i)  \, ; \, i\le j \right\},
$$
whereas the canonical orthonormal basis $\Psi_2$ of $\veth{2}$ 
can be identified with
$$
\left\{ \dfrac12 (W_i\otimes W_j - W_j\otimes W_i)  \, ; \, i<j\right\}.
$$

\end{proposition}

\begin{definition}[See \cite{abraham_marsden}, Definition 1.7.16]
If $\phi:\he n \to \he n$ is a diffeomorphism, then the push-forward
of a tensor field $t:=v\otimes \alpha$, with $v\in\otimes^r\mathfrak h$ and
$\alpha\in\otimes^s\mathfrak h^*$, then its push-forward $\phi_*t$
at a point $p\in\he n$ is defined as
$$
\phi_*t(p):= \big(\otimes_r ^s d\phi(\tilde p)\big)((v
\otimes
\alpha)(\tilde p)),
$$
where $\tilde p:=\phi^{-1}(p)$. Moreover, the pull-back $\phi^*t$ of $t$
is defined by
$$
\phi^*t = (\phi^{-1})_*.
$$

\end{definition}

A tensor field 
$$
v\otimes \alpha\in \Gamma\big(\he n, \big(\otimes^r\mathfrak h\big)  \otimes \big( \otimes^s\mathfrak h^*\big)\big)
$$
is said {\sl left invariant} if
$$
\left(\tau_{q}\right)_* v\otimes \alpha = v\otimes \alpha\quad \mbox{for all $q\in\he n$.}
$$

\begin{lemma}\label{left invar} Let $v\otimes \xi\in \big(\otimes^r\mathfrak h\big)_e  \otimes \big( \otimes^s\mathfrak h^*\big)_e $ be given.
If $p\in\he n$, we set 
$$\mc T_{p}(v\otimes \xi):= (\tau_{p})_*({v\otimes \xi})\in \big(\otimes^r\mathfrak h\big)_p  \otimes \big( \otimes^s\mathfrak h^*\big)_p.
$$
 Then
 the map $p\to \mc T_{p}(v\otimes \xi)$ is 
left-invariant.

Thus,  if $V\otimes W$ is a linear subspace of $ \big(\otimes^r\mathfrak h\big)_e  \otimes \big( \otimes^s\mathfrak h^*\big)_e $, then $\{\mc T_{p}(v\otimes \xi),\; v\otimes \xi
\in V\otimes W, \; p\in \he n\}$ defines a left-invariant subbundle of 
$ \big(\otimes^r\mathfrak h\big)  \otimes \big( \otimes^s\mathfrak h^*\big) $. In addition,
if $\{v_i\otimes\xi_j, \;i=1,\dots N, j=1,\dots M\}$ is a basis of $V\otimes W$, then
$\{\mc T_{p}(v\otimes \xi), \;i=1,\dots N, j=1,\dots M\}$ is a left-invariant basis of the
fiber over $p\in \he n$.
\end{lemma}

We remind the following well-known identity (see e.g. \cite{AMR}, Proposition 6.16):
\begin{remark}\label{ovvia} If $W\in \mathfrak h_1$ is identified with a first order differential operator and
$\phi: \he n\to \he n$ is a diffeomorphism, then
$$
W(u\circ \phi)(x) = \left [(\phi_*W) u\right ] (\phi (x)).
$$
Moreover, if $W,Z\in \mathfrak h_1$, then
$$
\phi_*(W\otimes Z) = \phi_* W\otimes \phi_* Z.
$$
\end{remark}

Set $N_h\! := \dim E_0^h$. Given a family of left-invariant bases $\{\xi_k^h,\; k=1,\dots N_h\}$ of $E_0^h$, $1\le h\le n$ as
in Lemma \ref{left invar}, the differential $d_c$ can be written ``in coordinates'' as follows.

\begin{proposition} If $\; 0\le h\le 2n$ and
$$\alpha=\sum_k \alpha_k\, \xi_k^h\in \Gamma(E_0^h),$$
then
$$d_c\alpha = \sum_{I,k}  P_{I,k}\alpha_k\, \xi_I^{h+1},$$
where 
\begin{enumerate}
\item[i)]  if $ h\neq n$, then the $P_{I,k}$'s are linear homogeneous polynomials in $W_1,\dots, W_{2n}\in \mathfrak h_1$ (that are identified 
with homogeneous with first
order left invariant horizontal differential operators), i.e.
$$
P_{I,k}=\sum_i F_{I,k,i} W_i,
$$
where the $F_{I,k,i} $'s are real constants;
\item[ii)]  if $\; h=n$, then then the $P_{I,k}$'s are linear homogeneous polynomials in 
$W_i\otimes W_j \in \otimes^2\mathfrak h_1$, $i,j=1,\dots,2n$
(that are identified 
with  homogeneous second order
 left invariant differential horizontal operators), i.e.
$$
P_{I,k}=\sum_{i,j}F_{I,k,i,j} W_i\otimes W_j,
$$
where the $F_{I,k,i,j} $'s are real constants.
\end{enumerate}
\end{proposition}

\begin{definition}\label{symbol} If $0\le h <n$ we denote $\sigma(d_c)$ the symbol of the intrinsic differential $d_c$
that is a smooth field of homomorphisms
$$
\sigma(d_c)\in \Gamma( \Hom (E_0^h, \mathfrak h_1\otimes E_0^{h+1}))
$$
defined as follows:  if $p\in\he{n}$, $\bar {\alpha} = \sum_k\bar{\alpha}_k \xi_k^h(p)\in (E_0^h)_p$, then 
we can assume there exists a smooth  differential 
form $\alpha = \sum_k\alpha_k  \xi_k^h \in \Gamma( E_0^h)$ such that $\bar{\alpha}= \alpha (p)$. Thus, if 
$u\in \mc E(\he{n})$ satisfies $u(p)=0$, then we have
$$
d_c(u\alpha)(p)= \sum_{I,k}  (P_{I,k}u)(p)\alpha_k(p)\xi_I^{h+1}(p).
$$
Hence we set
$$
\sigma(d_c)(p)\bar{\alpha}:= \sum_{I,k}  \bar{\alpha}_jP_{I,k}(p) \otimes \xi_I^{h+1}(p).
$$
If $h=n$,  $d_c$  is now a second order differential operator in the horizontal vector fields
and then its symbol $\sigma(d_c)$ can be identified with a section
$$
\sigma(d_c) \in \Gamma(\Hom( E_0^n, \otimes^2 \mathfrak h_1\otimes E_0^{n+1}))
$$
as follows:  if $p\in\he{n}$, $\bar {\alpha} = \sum_k\bar{\alpha}_k\, \xi_k^n(p)\in (E_0^n)_p$, then we 
we can assume there exists a smooth  differential 
form $\alpha = \sum_k\alpha_k  \, \xi_k^n \in E_0^n$ such that $\bar{\alpha}= \alpha (p)$. Thus, if 
$u\in \mc E(\he{n})$ satisfies $u(p)=0$ and $W_i u(p)=0$, $i=1,\dots,2n$, then we have
$$
d_c(u\alpha)(p)= \sum_{I,k}  (P_{I,k}u)(p)\alpha_k (p)\, \xi_I^{n+1}(p).
$$
Hence we set
$$
\sigma(d_c)(p)\bar{\alpha}:= \sum_{I,k } \bar{\alpha}_k P_{I,k}(p) \otimes \xi_I^{n+1}(p).
$$
On the other hand, the canonical projection
$$  p: \otimes^2h_1 \otimes E_0^{n+1} \to \dfrac{\otimes^2h_1 }{\covh 2} \otimes E_0^{n+1}
$$
given by 
$$
p(W_i\otimes W_j\otimes \xi):= [W_i\otimes W_j] \otimes \xi
$$
defines a new symbol (the symmetric part of the symbol)
$$
\Sigma(d_c):= p\circ \sigma(d_c) \in \Hom( E_0^n,  \dfrac{\otimes^2h_1 }{\covh 2} \otimes E_0^{n+1}).
$$
Clearly, since we are dealing with 2-tensors,  $[W_i\otimes W_j]$  can be represented by
$$
\frac12 \left( W_i\otimes W_j + W_j\otimes W_i\right).
$$
Thus
\begin{equation}\label{simmetrico}\begin{split}
\Sigma & (d_c) \bar{\alpha} 
\\&
= \dfrac12 \sum_I \sum_{i,j,k}\bar{\alpha}_k 
(F_{I,k,i,j} + F_{I,k,j,i}) \left(W_i\otimes W_j
 + W_j\otimes W_i \right) \otimes \xi_I^{n+1}
 \\&
=: \sum_I \sum_{i,j,k}\bar{\alpha}_k 
\tilde{F}_{I,k,i,j} \left(W_i\otimes W_j
 + W_j\otimes W_i \right) \otimes \xi_I^{n+1} .
\end{split}\end{equation}

\end{definition} 

\begin{remark} 
If $0\le h <n$, mimicking the usual definition of the principal symbol $\sigma(P)$ of a differential operator $P$ (see e.g. \cite{narashiman}, Definition 3.3.13 or \cite{palais}, IV.3), one can cook up a notion of horizontal principal symbol that takes into account Heisenberg homogeneity. For $d_c$, this map would belong to
$$
\Gamma(\Hom (E_0^h\otimes \mathfrak h_1^*, E_0^{h+1})).
$$
However, our notation is not misleading, since, by \cite{dieudonne} (16.8.2.3) and (16.18.3.4),
\begin{equation*}\begin{split}
 \Hom (E_0^h& \otimes \mathfrak h_1^*, E_0^{h+1}) \cong
\Hom( E_0^n, \Hom (\mathfrak h_1^*, E_0^{h+1})\;) 
\\&
\cong
\Hom( E_0^h, \mathfrak h_1\otimes E_0^{h+1}).
\end{split}\end{equation*}
An analogous comment applies when $h=n$. In this case, only the projection $\Sigma(d_c)$ of the symbol onto symmetric tensors will be used.

\end{remark}

 Since $ \mathfrak h_1$, $d_c$  and $E_0^*$ are  invariant under left translations, then the symbol $\sigma(d_c)$
is uniquely determined by its value at the point $p=e$. More precisely, we have:

\begin{proposition} If $1\le h < n$ and $p\in\he n$, 
then
the following diagram is commutative:
\begin{equation}\label{diagram 1}
\begin{CD}
(E^h_0)_p @>{\sigma(d_c)(p)}>>  (\mathfrak h)_p\otimes (E_0^{h+1})_p\\
@V{\otimes^h\tau_{p}(e)}VV   @A{\mc T_p}AA @.\\
 (E^h_0)_e@>{\sigma(d_c)(e)}>> (\mathfrak h)_e\otimes (E_0^{h+1})_e
\end{CD}
\end{equation}

An analogous comment applies when $h=n$.
\end{proposition}

\section{Analytic facts}
\label{af}

The proof of Theorem \ref{Hn} consists in applying
the following two results due to Chanillo \& Van Schaftingen, after an algebraic reduction that will be performed in the next section.

\begin{definition} \label{gradient} If $f:\he n\to \R$, we denote by
$\nabla_{\he{}}f$ the horizontal vector field
\begin{equation*}
\nabla_{\he{}}f:=\sum_{i=1}^{2n}(W_if)W_i,
\end{equation*}
whose coordinates are $(W_1f,...,W_{2n}f)$. If $\Phi$ is a horizontal vector field,
then $\nabla_{\he{}}\Phi$ is defined componentwise.
\end{definition}

\begin{theorem}[\cite{CvS2009}, Theorem 1]\label{chanillo_van} Let $\Phi\in \mc D(\he n,\mathfrak h_1)$ be a smooth
compactly supported horizontal vector field. Suppose  $G\in L^1_{\mathrm{loc}}(\he n,\mathfrak h_1)$
is $\he{}$-di\-ver\-gen\-ce free, i.e. if 
$$
G= \sum_i G_iW_i, \quad\mbox{then} \quad \sum_iW_iG_i=0 \quad\mbox{in $\mc D'(\he n)$}.
$$
Then
$$
\big|\scal{G}{\Phi}_{L^2(\he{n},\mathfrak h_1)} \big| \le C\|G\|_{L^1(\he{n},\mathfrak h_1)} \|\nabla_{\he{}} \Phi\|_{L^Q(\he{n},\mathfrak h_1)}.
$$

\end{theorem}

We notice that a stronger version of this result can be found in \cite{wang_yung}, Theorem 1.9.

\medskip

As in the Euclidean case, estimates similar to Theorem  \ref{chanillo_van} still hold when the condition on the divergence is replaced by a condition on higher-order derivatives \cite{van_schaftingen_JEMS}. 
Similar ideas have been applied in nilpotent homogeneous groups 
by S. Chanillo and J. Van Schaftingen as follows.

Let $k\ge 1$ be fixed, and let $G\in 
L^1(\he n, \otimes^k \mathfrak h_1 )$ belong to the space of horizontal $k$-tensors.
We can write
$$
G = \sum_{i_1,\dots,i_k} G_{i_1,\dots,i_k} W_{i_1}\otimes \cdots \otimes W_{i_k}.
$$
We remind that $G$ can be identified with the differential operator
$$
u\to Gu := \sum_{i_1,\dots,i_k} G_{i_1,\dots,i_k} W_{i_1}\cdots W_{i_k} u.
$$
Denoting  by $\mc D(\he n, \mathrm{Sym} (\otimes^k \mathfrak h_1 ))$ the subspace
of compactly supported smooth symmetric horizontal $k$-tensors,
we have:

\begin{theorem}[\cite{CvS2009}, Theorem 5]\label{chanillo_van5} Let $k\ge 1$ and
 $$G\in 
L^1(\he n, \otimes^k \mathfrak h_1 ), \quad \Phi \in \mc D(\he n, \mathrm{Sym} (\otimes^k\mathfrak h_1 )).
$$

Suppose
that
$$
\sum_{i_1,\dots,i_k}  W_{i_k}\cdots W_{i_1} G_{i_1,\dots,i_k} = 0 \quad\mbox{in $\mc D'(\he n)$}.
$$
Then
$$
\Big| \int_{\he n}\scal{\Phi}{G}\, dp \Big| \le C 
\|G\|_{L^1(\he n, \otimes^k \mathfrak h_1)} \|\nabla_{\he{}} \Phi\|_{L^Q(\he n, \otimes^k \mathfrak h_1)}.
$$
\end{theorem}

\section{Main algebraic step}
\label{main}

As in \cite{BF7}, our proof of Theorem \ref{Hn} relies  on  the fact 
(precisely stated in Theorem \ref{pierre} below) that the components 
with respect to a given basis of closed forms in $E_0^h$
can be viewed as the components of a horizontal vector field  with vanishing horizontal
divergence if $h\neq n,n+1$ or vanishing ``generalized horizontal divergence'' if
$h=n, n+1$.  More precisely, we have:

\begin{theorem}\label{pierre} Let $\alpha=\sum_J \alpha_J\xi_J^h\in \Gamma(E_0^h)$,
 $1\le h\le 2n$, be such that
$$
d_c\alpha = 0.
$$
Then
\begin{itemize}
\item if $ h \neq n$
then each component $\alpha_J$ of $\alpha$, $J=1,\dots, \dim E_0^h$, can be written as 
$$
\alpha_J = \sum_{I=1}^{\dim E_0^{h+1}} \sum_{i=1}^{2n} b_{i,I}^J G_{I,i},
$$
where the $b_{i,I}^J $'s are real constants and for any $I=1,\dots,\dim E_0^{h+1}$ the $G_{I,i}$'s are the components
of a horizontal vector  field
$$
G_I=\sum_i G_{I,i}W_i 
$$
with
$$
\sum_i W_i G_{I,i}=0, \quad I=1,\dots,\dim E_0^{{h+1}}.
$$
Moreover there exist a geometric constant $C>0$ such that for $I=1,\dots,\dim E_0^{{h+1}}$
and $1\le p\le\infty$
\begin{equation}\label{reverse 1}
\|G_I\|_{L^p(\he n, \veth{1})} \le C \|\alpha\|_{L^p(\he n, E_0^h)}.
\end{equation}

\item If $h=n$,
then each component $\alpha_J$ of $\alpha$, $J=1,\dots, \dim E_0^{n}$, can be written as 
$$
\alpha_J = \sum_{I=1}^{\dim E_0^{{n+1}}} \sum_{i,j} b_{i,j,I}^J (G_I^{\mathrm{Sym}})_{i,j}.
$$
Here the $b_{i,j,I}^J $'s are real constants and for any $I=1,\dots,\dim E_0^{n+1}$ the $(G_I^{\mathrm{Sym}})_{i,j}$'s are the components
of the symmetric part (see Proposition \ref{sym deco}) of the 2-tensor
$$
G_I=\sum_i G_{I,i,j}\, W_i \otimes W_j
$$
that satisfies
$$\sum_{i,j} W_i W_j \,G_{I,i,j}=0,
\quad I=1,\dots, \dim E_0^{n+1}.
$$
Moreover there exist a geometric constant $C>0$ such that for $I=1,\dots,\dim E_0^{{n+1}}$
and $1\le p\le\infty$
\begin{equation}\label{reverse 2}
\|G_I\|_{L^p(\he n, \otimes^2\mathfrak h_1)} \le C \|\alpha\|_{L^p(\he n, E_0^n)}.
\end{equation}
\end{itemize}
\end{theorem}

The proof of this theorem requires several
preliminary steps. First of all, 
we want to prove that the exterior differential $d_c$ is invariant (i.e. ``natural'') under the action
of a class of intrinsic transformation of $\he n$.

%

%
%

\begin{theorem}\label{natural} If $A$ belongs to the symplectic group $Sp_{2n}(\R)$, we associate with $A$
the real  $(2n+1)\times (2n+1)$ matrix
\begin{equation}
 \label{associated matrix}
f_A : \mathfrak h \to \mathfrak h, \quad f_A= \left( \begin{array}{cc} 
A_{2n\times 2n} &  0_{2n\times 1}  \\
0_{1\times 2n}  &  1      
\end{array} \right) \quad.
\end{equation}
Then
\begin{itemize}
\item[i)] $f_A(\mathfrak h_1) = \mathfrak h_1$;
\item [ii)] $f_A$ induces a homogeneous group isomorphism $\exp\circ f_A\circ \exp^{-1} $ still
denoted by $f_A$ such that
$f_A:\he n\to \he n$;
\item[iii)] $
f_A^* : \Gamma(E_0^*) \to \Gamma(E_0^*);
$
\item[iv)] for any $h$-form $\alpha\in \Gamma(E_0^h)$
$$d_c (f_A^*\alpha)= f_A^*(d_c \alpha);$$
\end{itemize}
\end{theorem}

\begin{proof} See \cite{rumin_pa} or  \cite{FT4}.

\end{proof}

\medskip


The next step consists in proving that the symbols $\sigma(d_c)$ and
$\Sigma(d_c)$ are injective. 

First of all, we remind that, by \cite{weil}, Chapter I, Theorem 3 and Corollary at p. 28), the following 
proposition holds:
\begin{proposition}\label{varia} Let $P^h$ be space of primitive forms defined in \eqref{al1}. Then
\begin{itemize}
\item[i)] if $1\le h \le 2n$, then the following orthogonal decomposition holds:
$$
\covh h = \bigoplus_{i\ge (n-h)^+} L^i (P^{h-2i});
$$
\item[ii)] $P^h = \{0\}$ if $h>n$;
\item[iii)] the map $L^{n-h}: \covh h \to \covh{2n-h}$ is a linear isomorphism;
\item[iv)] if $h\le n$, then $P^h = \ker L^{n-h+1}$;
\item[v)] the map $L^{n-h}: P^h \to \covh{2n-h}\cap \ker L$ is a linear isomorphism;
\item[vi)] a symplectic map $A\in Sp_{2n}(\mathbb R)$ commutes with $L$, i.e.
if $1\le h\le 2n-2$, then $[\otimes^h A, L]=0$.
\end{itemize}
\end{proposition}

The injectivity of the symbols will follow from the following result.

\begin{proposition}\label{irreducible 1} The symplectic group
$Sp_{2n}(\R)$ acts irreducibly on $P^h$ for $h=1,\dots,n$
and on $\ker L$ for $h=n+1,\dots,2n-1$.
\end{proposition}
\begin{proof}
If $1\le h\le n$, then the statement is proved in \cite{bourbaki_lie_7_8}, p. 203.
Suppose now $h > n$. If $A\in Sp_{2n}(\mathbb R)$ and $\alpha\in \ker L$, then 
by Proposition \ref{varia}, vi), $\otimes^hA L
\in \ker L$, so that $A$ acts on $\ker A$. On the other hand, if $V\subset \ker L\cap \covh h$
is invariant under the action of $Sp_{2n}(\mathbb R)$, then, by Proposition \ref{varia}, v) and vi),
$(L^{h-r})^{-1} V$ is also invariant under the action of $Sp_{2n}(\mathbb R)$, and the
assertion follows by the first part of the proof.
\end{proof}

Using Definition \ref{rumin in H}, we have:
\begin{corollary} \label{irreducible} If $A\in Sp_{2n}(\mathbb R)$ and $f_A$ is defined as in \eqref{associated matrix},
then $f_A$ acts irreducibly on
$E_0^h$ for $h=1,\dots,2n$.
\end{corollary}

Since $d_c$ is equivariant under all smooth contact transformations, it is in particular $Sp_{2n}(\R)$-equivariant. It follows that the kernels $\ker \sigma(d_c)$ and
$\ker \Sigma(d_c)$ are invariant subspaces for the action
of  $Sp_{2n}(\R)$, so that the injectivity will follow. In fact,
we have:

\begin{proposition}\label{invariance} Keeping in mind Definition \ref{symbol},
{if $1\le h \le 2n$ , $h\neq n$,} then $\ker \sigma (d_c)(e)$ is invariant under the action
of  $Sp_{2n}(\R)$, i.e. if $A\in Sp_{2n}(\R)$, 
 then we have:
$$
\mbox{if }\, \bar{\alpha}\in E_0^h \, \mbox{ and } \,\sigma(d_c)(e) (\bar{\alpha}) = 0
\quad\mbox{then}\quad
\sigma(d_c)(e) (( \otimes^h A)\ \bar{\alpha})=0.
$$
If $h=n$, then $\ker \Sigma (d_c)(e)$ is invariant under the action
of  $Sp_{2n}(\R)$, i.e., $A\in Sp_{2n}\R)$, 
 then we have:
$$
\mbox{if }\, \bar{\alpha}\in E_0^h \, \mbox{ and } \,\Sigma(d_c)(e) (\bar{\alpha}) = 0
\quad\mbox{then}\quad
\Sigma(d_c)(e) (( \otimes^h A)\ \bar{\alpha})=0.
$$
\end{proposition}

\begin{proof} Suppose first  $h < n$ and
let $\alpha\in \Gamma(E_0^h)$ be a differential form such that
$\alpha(e)=\bar{\alpha}$.
 Let $f_A$ be the  matrix associated with $A$ as in \eqref{associated matrix}.
We notice first that
\begin{equation}\label{9 jan}
f_A^*(\alpha)(e) = (\otimes^h A)\bar{\alpha},
\end{equation}
since $f_A(e)=e$.
Let now $u$
be a smooth function such that $u(e)=0$. We set $v:=u\circ f^{-1}_A$. Keeping again in mind that 
\eqref{9 jan}, 
we have also that $v(e)=0$.
By Theorem \ref{natural} and Remark \ref{ovvia}, we have:
\begin{equation}\label{nov 27: 1}\begin{split}
d_c & (u f_A^* \alpha)(e) 
= d_c(f_A^*(v \alpha))(e) = f_A^*(d_c(v\alpha))(e) 
\\&
=  (\otimes^{h+1}f_A)(d_c(v\alpha)(e))
 =  (\otimes^{h+1}A)(d_c(v\alpha)(e))
 \\&
 = \sum_{I,k}  (A^{-1} P_{I,k}u)(e)\bar{\alpha}_k (\otimes^{h+1}A)(\xi_I^{h+1}(e)).
\end{split}\end{equation}
Hence, by \eqref{9 jan}
\begin{equation}\label{nov 27: 2}\begin{split}
 \sigma & (d_c)(e) ((\otimes^h A)\bar{\alpha}) =  \sigma  (d_c)(e) ((f_A^*\alpha)(e))
\\& = 
\sum_I A^{-1} \left(\sum_{k}  \bar{\alpha}_k P_{I,k}(e)\right) \otimes (\otimes^{h+1}A)(\xi_I^{h+1}(e)).
\end{split}\end{equation}
On the other hand, by assumption,
$$
\sum_{I,k}  \bar{\alpha}_k P_{I,k}(e) \otimes \xi_I^{h+1}(e)=0,
$$
so that
$$
\sum_{k}  \bar{\alpha}_kP_{I,k}(e)=0 \quad I=1,\dots,\dim E_0^{h+1},
$$
since the $\xi_I^{h+1}$'s are linearly independent.
Thus eventually from \eqref{nov 27: 2}
$$
\sigma(d_c)(e) (( \otimes^h A)\ \bar{\alpha})=0.
$$

Consider now the case $h=n$, when $d_c$ is a second order operator in the
horizontal derivatives. We stress that $E_0^{n+1}$ contains only vertical forms, i.e.
forms that are multiple of the contact form $\theta$. Suppose  $ \Sigma(d_c)(e) \bar{\alpha}=0$.
Then, by \eqref{simmetrico},
\begin{equation}\label{per componenti k=n}
\sum_{i,j,k}\bar{\alpha}_k 
\tilde{F}_{I,k,i,j} \left(W_i\otimes W_j
 + W_j\otimes W_i \right) = 0, \quad I=1, \dots, \, \dim E_0^{h+1}.
\end{equation}

We take now
$u\in \mc E(\he{n})$ satisfies $u(e)=0$ and $W_i u(e)=0$, $i=1,\dots,2n$. As above, we set $v:=u\circ f_A^{-1}$. Keeping in mind that $f_A(e)=e$, 
we have also that $v(e)=0$ and $(W_i v)(e)=0$, $i=1,\dots,2n$.   Then equations \eqref{nov 27: 1} become
\begin{equation*}\begin{split}
d_c & (u f_A^* \alpha)(e) 
= d_c(f_A^*(v \alpha))(e) = f_A^*(d_c(v\alpha))(e) 
\\&
=  (\otimes^{h+1}f_A)(d_c(v\alpha)(e))
 \\&
 = \sum_{I,k}\sum_{i,j}F_{I,k,i,j} (A^{-1}W_i )(A^{-1}W_j)  u(e) \bar{ \alpha}_k(\otimes^{n+1}f_A)(\xi_I^{n+1}(e)).
\end{split}\end{equation*}
Therefore, keeping in mind Remark \ref{ovvia} and \eqref{9 jan}, we have:
\begin{equation*}\begin{split}
 \Sigma & (d_c)(e) ((\otimes^h A)\bar{\alpha}) =  \Sigma  (d_c)(e) ((f_A^*\alpha)(e))
\\& = 
\sum_{I}\sum_{i,j,k} \bar{ \alpha}_k \tilde{F}_{I,k,i,j} \left(
(A^{-1}W_i) (e)\otimes (A^{-1}W_j )(e)\right.
\\&\hphantom{xxxxxxxxx}
\left. + (A^{-1}W_j) (e)\otimes (A^{-1}W_i )(e)
\right) \otimes (\otimes^{n+1}f_A)(\xi_I^{n+1}(e))
\\& = 
\sum_{I}  A^{-1}\Big(\sum_{i,j,k} \bar{ \alpha}_k \tilde{F}_{I,k,i,j}  \big(W_i (e)\otimes W_j(e)
\\&  \hphantom{xxxxxxxxxxx}+ W_j(e)\otimes W_i(e)\big) \Big)\otimes
 (\otimes^{n+1}f_A)(\xi_I^{n+1}(e))
 = 0,
\end{split}\end{equation*}
by \eqref{per componenti k=n}.  

Finally, the proof for $h>n$ can be carried out precisely as in the case $h<n$, with only minor changes. In particular,
\eqref{9 jan} must be replaced taking into account that a form $\alpha\in E_0^h$ has also a vertical
component of the form $\beta\wedge\theta$ and that
\begin{equation*}
f_A^*(\beta\wedge\theta)(e) = (\otimes^h A)\beta (e) \wedge \theta.
\end{equation*}

This completes the proof of the proposition.

\end{proof}

\begin{proof}[Proof of Theorem \ref{pierre}] First of all, we notice that, by Corollary \ref{irreducible} and 
Proposition \ref{invariance},
both $\ker \sigma(d_c)(e) $ (if $h \neq n$) and $\ker \Sigma(d_c)(e)$ (if $h=n$) are the null space $\{0\}$, 
and hence both $ \sigma(d_c)(e)$ and $\Sigma(d_c)(e)$ have a left inverse
$$B_h\in \mathrm{Hom}\, ( \mathfrak h_1\otimes
(E_0^{h+1})_e, (E_0^h)_e) \quad\mbox{if $h \neq n$.}
$$
and 
$$B_n\in \mathrm{Hom}\, (\mathrm{Sym}\, (\otimes^2\mathfrak h_1)\otimes
(E_0^{n+1})_e, (E_0^n)_e) \quad\mbox{if $h=n$.}
$$
By the commutativity of the diagram \eqref{diagram 1}, $B_h$ and $B_n$ can be identified with constant coefficient
maps
$$
B_h\in \mathrm{Hom}\, ( \mathfrak h_1\otimes
E_0^{h+1}, E_0^h), \quad B_n\in \mathrm{Hom}\, (\mathrm{Sym}\, (\otimes^2\mathfrak h_1)\otimes
E_0^{n+1}, E_0^n)
$$
such that
\begin{equation}\label{inverso}
\alpha = B_h(\sigma(d_c)\alpha) \quad\mbox{for all $\alpha\in \Gamma(E_0^h)$, $h\neq n$,}
\end{equation}
and
\begin{equation}\label{inverso bis}
\alpha = B_n(\Sigma(d_c)\alpha) \quad\mbox{for all $\alpha\in \Gamma(E_0^n)$.}
\end{equation}

We deal first with the case $h \neq n$ and we set
$$
B_h(W_i\otimes \xi_I^{h+1}):= \sum_J b_{i, I}^J \xi_J^h.
$$
Then, if we write $\alpha=\sum_J \alpha_J\xi_J^h\in \Gamma(E_0^h)$
and 
$$
P_{I,k}=\sum_i F_{I,k,i} W_i,
$$
(where the $F_{I,k,i} $'s are real constants) identity \eqref{inverso} becomes
\begin{equation}\label{inverso coord}\begin{split}
\alpha &= B_h(\sum_{I,k}\alpha_k P_{I,k}\otimes \xi_I^{h+1})
\\&=
\sum_J \left(\sum_{I,k,i} b_{i,I}^J F_{I,k,i} \alpha_k\right)\xi_J^h,
\end{split}\end{equation}
so that
\begin{equation}\label{inverso comp}\begin{split}
\alpha_J = \sum_{I,k,i} b_{i,I}^J F_{I,k,i} \alpha_k,
\quad \mbox{for $J=1,\dots, \dim E_0^h$.}
\end{split}\end{equation}

Suppose now $d_c\alpha =0$. Then, writing the identity in coordinates, if
$I=1,\dots, \dim E_0^{h+1}$, we have
$$\sum_i W_i  \left( \sum_{k} F_{I,k,i} \alpha_k\right)=0,
$$
so that, if we denote by 
$G_I$ 
the horizontal vector field 
\begin{equation*}\begin{split}
G_I=\sum_i G_{I,i}W_i :=\sum_i  \left( \sum_{k} F_{I,k,i} \alpha_k\right) W_i
\end{split}\end{equation*}
then
$$
\sum_i W_i G_{I,i}  = 0, \qquad I=1,\dots, \dim E_0^{h+1}.
$$
Thus \eqref{inverso comp} reads as
$$
\alpha_J = \sum_{I,i} b_{i,I}^J G_{I,i},
$$
achieving the proof in the case $h<n$.

We deal now with the case $h=n$ and we set
$$
B_n((W_i\otimes W_j+W_j\otimes W_i)\otimes \xi_I^{n+1}):= \sum_J b_{i,j, I}^J \xi_J^n.
$$
Thus, if we write $\alpha=\sum_J \alpha_J\xi_J^h\in \Gamma(E_0^n)$
by \eqref{simmetrico}, identity \eqref{inverso bis} becomes
\begin{equation}\label{inverso coord bis}\begin{split}
\alpha &= B_n(\sum_{I,k,i,j}\alpha_k \tilde{F}_{I,k,i,j} \left(W_i\otimes W_j
 + W_j\otimes W_i \right) \otimes \xi_I^{n+1})
\\&=
\sum_J \left(\sum_{I,k,i,j} b_{i,j,I}^J \tilde{F}_{I,k,i,j} \alpha_k\right)\xi_J^n,
\end{split}\end{equation}
so that
\begin{equation}\label{inverso comp bis}\begin{split}
\alpha_J = \sum_{I,k,i, j} b_{i,j,I}^J \tilde{F}_{I,k,i,j} \alpha_k,
\quad \mbox{for $J=1,\dots, \dim E_0^n$.}
\end{split}\end{equation}

Denote by 
$G_I$ 
the horizontal tensor field 
\begin{equation}\label{slucia 1}\begin{split}
G_I &=\sum_{i,j} G_{I,i,j}W_i\otimes W_j :=\sum_{i,j}  \left( \sum_{k} F_{I,k,i,j} \alpha_k\right) 
W_i\otimes W_j.
\end{split}\end{equation}
By Proposition \ref{sym deco} we can write
$$
G_I = G_I^{\mathrm{Sym}}  + G_I^{\mathrm{Skew}}  ,
$$
where
\begin{equation*}\label{slucia 2}\begin{split}
G_I^{\mathrm{Sym}} &=
\sum_{i,j} (G_I^{\mathrm{Sym}})_{i,j} (W_i\otimes W_j
 + W_j\otimes W_i) 
 \\& :=
\sum_{i,j}  \left( \sum_{k} \tilde{F}_{I,k,i,j} \alpha_k\right) 
(W_i\otimes W_j
 + W_j\otimes W_i) .
\end{split}\end{equation*}

We suppose now that $d_c\alpha =0$, that, by \eqref{slucia 1},
in coordinates is
$$\sum_{i,j} W_i W_j \,G_{I,i,j}=0,
\quad I=1,\dots, \dim E_0^{n+1}.
$$

Thus \eqref{inverso comp bis} reads as
$$
\alpha_J = \sum_{I,i,j} b_{i,j,I}^J (G_I^{\mathrm{Sym}})_{i,j}
\quad \mbox{for $J=1,\dots, \dim E_0^n$,}
$$
achieving the proof in the case $h=n$.

\end{proof}

\section{Proof of Theorem \ref{Hn}}
\label{analysis}

The proof follows the lines of \cite{BF7}. Let us reming few facts of harmonic
analysis in homogeneous groups.

A differential operator $P:\Gamma(E_0^h)\to \Gamma (E_0^k)$ is said left-invariant
if for all $q\in\he n$
$$
P( \tau_q)_* \alpha = (\tau_q)_*(P\alpha)\quad \mbox{for all $\alpha\in \Gamma(E_0^h)$.}
$$
    If $f$ is a real function defined in $\he n$, we denote
    by $\ccheck f$ the function defined by $\ccheck f(p):=
    f(p^{-1})$, and, if $T\in\mc D'(\he n)$, then $\ccheck T$
    is the distribution defined by $\Scal{\ccheck T}{\phi}
    :=\Scal{T}{\ccheck\phi}$ for any test function $\phi$.
    
    Following e.g. \cite{folland_stein}, we can define a group
convolution in $\he n$: if, for instance, $f\in\mc D(\he n)$ and
$g\in L^1_{\mathrm{loc}}(\he n)$, we set
\begin{equation}\label{group convolution}
f\ast g(p):=\int f(q)g(q^{-1}\cdot p)\,dq\quad\mbox{for $q\in \he n$}.
\end{equation}
We remind that, if (say) $g$ is a smooth function and $P$
is a left invariant differential operator, then
$$
P(f\ast g)= f\ast Pg.
$$
We remind also that the convolution is again well defined
when $f,g\in\mc D'(\he n)$, provided at least one of them
has compact support. In this case the following identities
hold
\begin{equation}\label{convolutions var}
\Scal{f\ast g}{\phi} = \Scal{g}{\ccheck f\ast\phi}
\quad
\mbox{and}
\quad
\Scal{f\ast g}{\phi} = \Scal{f}{\phi\ast \ccheck g}
\end{equation}
 for any test function $\phi$.
 
As in \cite{folland_stein},
we also adopt the following multi-index notation for higher-order derivatives. If $I =
(i_1,\dots,i_{2n+1})$ is a multi--index, we set  
$W^I=W_1^{i_1}\cdots
W_{2n}^{i_{2n}}\;T^{i_{2n+1}}$. 
By the Poincar\'e--Birkhoff--Witt theorem, the differential operators $W^I$ form a basis for the algebra of left invariant
differential operators in $\he n$. 
Furthermore, we set 
$|I|:=i_1+\cdots +i_{2n}+i_{2n+1}$ the order of the differential operator
$W^I$, and   $d(I):=i_1+\cdots +i_{2n}+2i_{2n+1}$ its degree of homogeneity
with respect to group dilations.

 Suppose now $f\in\mc E'(\he n)$ and $g\in\mc D'(\he n)$. Then,
 if $\psi\in\mathcal D(\he n)$, we have
 \begin{equation}\label{convolution by parts}
 \begin{split}
\Scal{(W^If)\ast g}{\psi}&=
 \Scal{W^If}{\psi\ast \ccheck g} =
  (-1)^{|I|}  \Scal{f}{\psi\ast (W^I \,\ccheck g)} \\
&=
 (-1)^{|I|} \Scal{f\ast \ccheck W^I\,\ccheck g}{\psi}.
\end{split}
\end{equation}

\medskip

Following \cite{folland}, we remind now the notion 
of {\it kernel of order $a$}, as well as some basic properties.

\begin{definition} A kernel of order $a$ is a 
homogeneous distribution of degree $a-Q$
(with respect to group dilations),
that is smooth outside of the origin.
\end{definition}

\begin{proposition}\label{kernel}
Let $K\in\mc D'(\he n)$ be a kernel of order $a$.
\begin{itemize}
\item[i)] $\ccheck K$ is again a kernel of order $a$;
\item[ii)] $W_\ell K$ is a a kernel of order $a-1$ for
any horizontal derivative $W_\ell K$, $\ell=1,\dots,2n$;
\item[iii)]  If $a>0$, then $K\in L^1_{\mathrm{loc}}(\he n)$.
\end{itemize}
\end{proposition}

\begin{definition}\label{rumin laplacian} 
In $\he n$, following \cite{rumin_jdg}, we define
the operator $\Delta_{\he{},h}$  on $E_0^h$ by setting
\begin{equation*}
\Delta_{\he{},h}=
\left\{
  \begin{array}{lcl}
     d_c\delta_c+\delta_c d_c\quad &\mbox{if } & h\neq n, n+1;
     \\ (d_c\delta_c)^2 +\delta_cd_c\quad& \mbox{if } & h=n;
     \\d_c\delta_c+(\delta_c d_c)^2 \quad &\mbox{if }  & h=n+1.
  \end{array}
\right.
\end{equation*}

\end{definition}

Notice that $-\Delta_{\he{},0} = \sum_{j=1}^{2n}(W_j^2)$ is the usual sub-Laplacian of
$\he n$.

 For sake of simplicity, once a basis  of $E_0^h$
is fixed, the operator $\Delta_{\he{},h}$ can be identified with a matrix-valued map, still denoted
by $\Delta_{\he{},h}$
\begin{equation}\label{matrix form}
\Delta_{\he{},h} = (\Delta_{\he{},h}^{ij})_{i,j=1,\dots,N_h}: \mc D'(\he{n}, \rn{N_h})\to \mc D'(\he{n}, \rn{N_h}).
\end{equation}

This identification makes possible to avoid the notion of currents: we refer to \cite{BFTT} for
this more elegant presentation.

Combining \cite{rumin_jdg}, Section 3,   and \cite{BFT3}, Theorems 3.1 and 4.1, we obtain the following result.

\begin{theorem}\label{global solution}
If $0\le h\le 2n+1$, then the differential operator $\Delta_{\he{},h}$ is
hypoelliptic of order $a$, where $a=2$ if $h\neq n, n+1$ and  $a=4$ 
if $h=n, n+1$ with respect to group dilations. Then
\begin{enumerate}
\item[i)] for $j=1,\dots,N_h$ there exists
\begin{equation}\label{numero}
    K_j =
\big(K_{1j},\dots, K_{N_h j}\big), \quad j=1,\dots N_h
\end{equation}
 with $K_{ij}\in\mc D'(\he{n})\cap \mc
E(\he{n} \setminus\{0\})$,
$i,j =1,\dots,N$;
\item[ii)] if $a<Q$, then the $K_{ij}$'s are
kernels of type $a$
 for
$i,j
=1,\dots, N_h$

 If $a=Q$,
then the $K_{ij}$'s satisfy the logarithmic estimate
$|K_{ij}(p)|\le C(1+|\ln\rho(p)|)$ and hence
belong to $L^1_{\mathrm{loc}}(\he{n})$.
Moreover, their horizontal derivatives  $W_\ell K_{ij}$,
$\ell=1,\dots,2n$, are
kernels of type $Q-1$;
\item[iii)] when $\alpha\in
\mc D(\he{n},\rn {N_h})$,
if we set
\begin{equation}\label{numero2}
    \mc K\alpha:=
\big(    
    \sum_{j}\alpha_j\ast  K_{1j},\dots,
     \sum_{j}\alpha_j\ast  K_{N_hj}\big),
\end{equation}
 then $ \Delta_{\he{},h}\mc K\alpha =  \alpha. $
Moreover, if $a<Q$, also $\mc K\Delta_{\he{},h} \alpha =\alpha$.

\item[iv)] if $a=Q$, then for any $\alpha\in
\mc D(\he{n},\rn {N_h})$ there exists 
$\beta_\alpha:=(\beta_1,\dots,\beta_{N_h})\in \rn{N_h}$,  such that
$$\mc K \Delta_{\he{},h}\alpha - \alpha = \beta_\alpha.$$

%
\end{enumerate}
\end{theorem}

\begin{remark}\label{K}
Coherently with formula \eqref{matrix form}, the operator $\mc K$ can be identified
with an operator (still denoted by $\mc K$) acting on smooth compactly
supported differential forms in $\mc D(\he n, E_0^h)$.
\end{remark}

\begin{proof}[Proof of Theorem \ref{Hn}]

The case $h=0$ is well known (\cite{FGaW}, \cite{CDG}, \cite{MSC}).
  
  \medskip
  
\noindent \textbf{Case}  $\mathbf{1< h< 2n}$ \textbf{and} $\mathbf{h \neq n, n+1}.$
 If $u, \phi \in \mc D(\he n, E_0^h)$, we can write
\begin{equation}\label{representation 2}\begin{split}
\scal{u}{\phi}_{L^2(\he n,E_0^h)} & = \scal{ u}{\Delta_{\he{},h}\mc K \phi}_{L^2(\he n,E_0^h)}
\\&
= \scal{ u}{(\delta_cd_c + d_c\delta_c)\mc K\phi}_{L^2(\he n,E_0^h)}.
\end{split}
\end{equation}
Consider now the first term in the previous sum, 
$$
 \scal{ u}{\delta_cd_c \mc K\phi}_{L^2(\he n,E_0^h)}=  \scal{d_c u}{d_c \mc K\phi}_{L^2(\he n,E_0^{h+1})}.
$$
If we write $f:= d_c  u$, then $d_cf=0$.
From now on, without loss of generality, for $1\le h\le 2n+1$ we take an orthonormal basis of $E_0^h$, still denoted by  
$\{\xi_\ell^{h}\; ; \; \ell =1,\dots, \dim E_0^h\}$.
Thus, since
$f, d_c \mc K\phi\in E_0^{h+1}$, we can write $f=\sum_{\ell=1}^{\dim E_0^{h+1}} f_\ell\xi_\ell^{h+1}$,
$d_c \mc K\phi=\sum_{\ell=1}^{\dim E_0^{h+1}} (d_c \mc K\phi)_\ell\xi_\ell^{h+1}$,
and hence
we can reduce ourselves to estimate
\begin{equation}\label{toprove}
\scal{f_\ell}{(d_c \mc K\phi)_\ell}_{L^2(\he n)}\quad\mbox{for $\ell=1,\dots,\dim E_0^{h+1}$.}
\end{equation}

By Theorem \ref{pierre}, if $h \neq n-1$ ,  each component $f_\ell$ of $f$ can be written as 
$$
f_\ell = \sum_{I=1}^{\dim E_0^{{h+2}}} \sum_{i=1}^{2n} b_{i,I}^\ell G_{I,i},
$$
where the $b_{i,I}^\ell $'s are real constants and for any $I=1,\dots,\dim E_0^{{h+2}}$ the $G_{I,i}$'s are the components of
an horizontal vector field
$$
G_I=\sum_i G_{I,i}W_i 
$$
with 
\begin{equation}\label{Div 1}
\sum_i W_i G_{I,i}  = 0, \qquad I=1,\dots, \dim E_0^{{h+2}}.
\end{equation}
On the other hand, for $h=n-1$, 
each component $f_\ell$ of $f$, $\ell=1,\dots, \dim E_0^{n}$, can be written as 
$$
f_\ell = \sum_{I=1}^{\dim E_0^{{n+1}}} \sum_{i,j} b_{i,j,I}^\ell (G_I^{\mathrm{Sym}})_{i,j}.
$$
Here the $b_{i,j,I}^\ell $'s are real constants and for any $I=1,\dots,\dim E_0^{n+1}$ the $(G_I^{\mathrm{Sym}})_{i,j}$'s are the components
of the symmetric part (see Proposition \ref{sym deco}) of the 2-tensor
$$
G_I=\sum_i G_{I,i,j}\, W_i \otimes W_j
$$
that satisfies
\begin{equation}\label{Div 2}
\sum_{i,j} W_i W_j \,G_{I,i,j}=0,
\quad I=1,\dots, \dim E_0^{n+1}.
\end{equation}
Suppose now $h \neq n-1$. In order to estimate the terms of \eqref{toprove}, we have to estimate
terms of the form
\begin{equation}\label{toprove1}
\scal{G_{I,i}}{(d_c \mc K\phi)_\ell}_{L^2(\he n)} =
\scal{G_{I}}{\Phi}_{L^2(\he n,\mathfrak h_1)},
\end{equation}
where
$$
\Phi = 
(d_c \mc K\phi)_\ell  W_i.
$$
We can apply Theorem \ref{chanillo_van}. Keeping in mind \eqref{reverse 1}, we obtain
\begin{equation}\label{4.2:1} 
\left|\scal{f_\ell}{(d_c \mc K\phi)_\ell}_{L^2(\he n)}\right| \le C\|f\|_{L^1(\he n, E_0^{h+1})}
\| \nabla_{\mathbb H} d_c \mc K\phi\|_{L^Q(\he n, E_0^{h+1})}.
\end{equation}
On the other hand, $\nabla_{\he{}}d_c \mc K\phi$ can be expressed as a sum of terms 
with components of the form
$$
\phi_j\ast W^I\tilde K_{ij}\, \quad\mbox{with $d(I)=2$.}
$$
By Theorem \ref{global solution}, iv) and Proposition \ref{kernel}, ii) $W^I\tilde K_{ij}$ are  kernels
of type 0, so that, by \cite{folland}, Proposition 1.9 we have
\begin{equation}\label{10.2:1}
|\scal{f_\ell}{(d_c \mc K\phi)_\ell}_{L^2(\he n)}| \le  C\|f \|_{L^1(\he{n},E_0^{h+1})}\|\phi\|_{L^{Q}(\he{n},E_0^h)}.
\end{equation}
The same argument can be carried out for all the components of $f$, yielding
\begin{equation}\label{27.2:4}
|\scal{f}{d_c \mc K\phi}_{L^2(\he n, E_0^{h+1})}| \le  C\|f \|_{L^1(\he{n},E_0^{h+1})}\|\phi\|_{L^{Q}(\he{n},E_0^h)}.
\end{equation}

Suppose now $h=n-1$ then we have to estimate terms of the form 
\begin{equation}\label{toprove2}
\scal{(G_I^{\mathrm{Sym}})_{i,j}}{(d_c \mc K\phi)_\ell}_{L^2(\he n)} =
\scal{G_{I}}{\Phi}_{L^2(\he n,\otimes^2\mathfrak h_1)},
\end{equation}
$$
\Phi = 
(d_c \mc K\phi)_\ell  \big(W_i \otimes W_j + W_j \otimes W_i)
\in \Gamma (\he n, \mathrm{Sym}\, (\otimes^2\mathfrak h_1)).
$$
We can apply Theorem \ref{chanillo_van5} and we obtain 
\begin{equation}\label{27.2:5}
|\scal{f}{d_c \mc K\phi}_{L^2(\he n, E_0^{n})}| \le  C\|f \|_{L^1(\he{n},E_0^{n})}\|\phi\|_{L^{Q}(\he{n},E_0^{n-1})}.
\end{equation}

This achieve the estimate of the first term of \eqref{representation 2} for all 
$1<h<2n$, $h\neq n, n+1$.
\begin{equation}\label{11.2:1}
|\scal{u}{\delta_c d_c \mc K\phi}_{L^2(\he n, E_0^h)}| \le  C\|f \|_{L^1(\he{n},E_0^{h+1})}\|\phi\|_{L^{Q}(\he{n},E_0^h)}.
\end{equation}

Consider now the second term in \eqref{representation 2}
\begin{equation*}\begin{split}
&\scal{u}{ d_c\delta_c \mc K  \phi}_{L^2(\he n,E_0^h)} = \scal{ \delta_c u}{\delta_c \mc K\phi}_{L^2(\he n,E_0^{h-1})}
\\&\hphantom{xxxxx} 
= \scal{g}{\delta_c \mc K\phi}_{L^2(\he n,E_0^{h-1})} = \scal{\ast g}{\ast \delta_c \mc K\phi}_{L^2(\he n,
E_0^{2n+2-h})},
\end{split}\end{equation*}
where $\ast$ denotes the Hodge duality. We notice now that from $\delta_cu =g$,
by Hodge duality we have $d_c\ast u = \ast g$. Hence $d_c (\ast g)=0$, and thus, arguing
precisely as above, we get
\begin{equation}\label{11.2:2}\begin{split}
\left| \scal{u}{ d_c\delta_c \mc K  \phi}_{L^2(\he n,E_0^h)} \right|
\le  C\|g \|_{L^1(\he{n},E_0^{h+1})}\|\phi\|_{L^{Q}(\he{n},E_0^h)}.
\end{split}\end{equation}

Combining \eqref{11.2:2} with  \eqref{11.2:1}, we get eventually
$$
|\scal{u}{\phi}_{L^2(\he n, E_0^h)} |  \le C\big(\|f\|_{L^1(\he n, E_0^{h+1})} + 
\|g\|_{L^1(\he n,E_0^{h-1})} \big) \|\phi\|_{L^Q(\he n, E_0^h)},
$$
and hence
$$
\|u\|_{L^{Q/(Q-1)}(\he{n}, E_0^h)}\le C\big( \|f \|_{L^1(\he{n}, E_0^{h+1})} + 
\|g\|_{L^1(\he n,E_0^{h-1})} \big).
$$
This completes the proof of statement iii) of the theorem.

\noindent \textbf{Case}  $\mathbf{h=1, 2n}$. By Hodge duality we may
restrict ourselves to the case $h=1$. Again we write
\begin{equation}\label{representation 2bis}\begin{split}
\scal{u}{\phi}_{L^2(\he n,E_0^1)} & = \scal{ u}{\Delta_{\he{},h}\mc K \phi}_{L^2(\he n,E_0^1)}
\\&
= \scal{ u}{\delta_cd_c \mc K\phi}_{L^2(\he n,E_0^1)}
+ \scal{ u}{d_c\delta_c\mc K\phi}_{L^2(\he n,E_0^1)}
\\&
= \scal{ f}{d_c \mc K\phi}_{L^2(\he n,E_0^1)}
+ \scal{ g}{\delta_c\mc K\phi}_{L^2(\he n,E_0^1)}.
\end{split}
\end{equation}
In order to estimate the first term $\scal{ f}{d_c \mc K\phi}_{L^2(\he n,E_0^1)}$,
we repeat verbatim the arguments above for the corresponding term in the case $h\neq n, n+1$.
As for the second term,
by Theorem \ref{global solution}, formula \eqref{numero2},
and keeping in mind that $\delta_c$  is an operator of order 1 in the horizontal derivatives
when acting on $E_0^1$
the quantity $\delta_c \mc K\phi$ can be written as a sum of terms such as
$$
\phi_j\ast W_\ell \tilde K_{ij}, \quad\mbox{with  $\ell=1,\dots,2n$.}
$$
On the other hand,
$$
\scal{g}{\phi_j\ast W_\ell \tilde K_{ij} }_{L^2(\he n)}
= \scal{g \ast \ccheck(W_\ell \tilde K_{ij}) }{\phi_j }_{L^2(\he n)}
$$

Notice the $W_\ell \tilde K_{ij}$'s  and hence the $\ccheck(W_\ell \tilde K_{ij})$'s are  kernels of type 1.
Thus, by Theorem 6.10 in \cite{folland_stein},
$$
|\scal{ g}{\phi_j\ast W_\ell \tilde K_{ij} }_{L^2(\he n)} |
\le C \|g\|_{\mc H^1(\he n)} \|\phi\|_{L^Q(\he n, E_0^1)}.
$$
Combining this estimate with the one in \eqref{11.2:1}, we get eventually
$$
|\scal{u}{\phi}_{L^2(\he n, E_0^1)} |  \le C\big(\|f\|_{L^1(\he n, E_0^2)} + 
\|g\|_{\mc H^1(\he n)} \big) \|\phi\|_{L^Q(\he n, E_0^1)},
$$
and hence
$$
\|u\|_{L^{Q/(Q-1)}(\he{n}, E_0^1)}\le C\big( \|f \|_{L^1(\he{n}, E_0^2)} + \|g \|_{\mc H^1(\he {n})}\big).
$$
This completes the proof of statement ii) of the theorem.

\noindent \textbf{Case}  $\mathbf{h=n,n+1}$. By Hodge duality we may
restrict ourselves to the case $h=n$. 

 If $u,\phi\in E_0^n$ are smooth compactly supported
forms, then we can write
\begin{equation}\label{representation 3}
\begin{split}
\scal{u}{\phi}_{L^2(\he n,E_0^n)} & = \scal{ u}{\Delta_{\he{},n}\mc K \phi}_{L^2(\he n,E_0^n)}
\\&
= \scal{ u}{(\delta_cd_c + (d_c\delta_c)^2)\mc K\phi}_{L^2(\he n,E_0^n)}.
\end{split}
\end{equation}
Consider now the term
$$
 \scal{ u}{\delta_cd_c \mc K\phi}_{L^2(\he n,E_0^n)}=  \scal{d_c u}{d_c \mc K\phi}_{L^2(\he n,E_0^{n+1})}
 =  \scal{f}{d_c \mc K\phi}_{L^2(\he n,E_0^{n+1})}.
$$
Let us write $f:= d_c  u$. 
Again, $d_c f=0$, and hence, as above,
if  $f=\sum_{\ell=1}^{\dim E_0^{n+1}} f_\ell\xi_\ell^{n+1}$,
$d_c \mc K\phi=\sum_{\ell=1}^{\dim E_0^{n+1}} (d_c \mc K\phi)_\ell\xi_\ell^{n+1}$,
and thus
we can reduce ourselves to estimate
\begin{equation}\label{toprove bis}
\scal{f_\ell}{(d_c \mc K\phi)_\ell}_{L^2(\he n)}\quad\mbox{for $\ell=1,\dots,\dim E_0^{n+1}$.}
\end{equation}

By Theorem \ref{pierre}, each component $f_\ell$ of $f$ can be written as 
$$
f_\ell = \sum_{I=1}^{\dim E_0^{{n+2}}} \sum_{i=1}^{2n} b_{i,I}^\ell G_{I,i},
$$
where the $b_{i,I}^\ell $'s are real constants and for any $I=1,\dots,\dim E_0^{{n+2}}$ the $G_{I,i}$'s are the components of
an horizontal vector field
$$
G_I=\sum_i G_{I,i}W_i 
$$
with 
\begin{equation}\label{Div 1 bis}
\sum_i W_i G_{I,i}  = 0, \qquad I=1,\dots, \dim E_0^{{n+2}}.
\end{equation}

As in the previous cases, in order to estimate the terms of \eqref{toprove bis}, we have to deal
terms of the form
\begin{equation}\label{toprove1 bis}
\scal{G_{I,i}}{(d_c \mc K\phi)_\ell}_{L^2(\he n)} =
\scal{G_{I}}{\Phi}_{L^2(\he n,\mathfrak h_1)},
\end{equation}
where
$$
\Phi = 
(d_c \mc K\phi)_\ell  W_i.
$$
We can apply Theorem \ref{chanillo_van}. Again keeping in mind \eqref{reverse 1}, we obtain
\begin{equation}\label{27.2:1} 
\left|\scal{f_\ell}{(d_c \mc K\phi)_\ell}_{L^2(\he n)}\right| \le C\|f\|_{L^1(\he n, E_0^{n+1})}
\| \nabla_{\mathbb H} d_c \mc K\phi\|_{L^Q(\he n, E_0^{n+1})}.
\end{equation}
On the other hand, $\nabla_{\he{}}d_c \mc K\phi$ can be expressed as a sum of terms 
with components of the form
$$
\phi_j\ast W^I\tilde K_{ij}\, \quad\mbox{with $d(I)=3$,}
$$
since the differential $d_c$ on $n$-forms has order 2 in the horizontal derivatives.
By Theorem \ref{global solution}, iv) and Proposition \ref{kernel}, ii) $W^I\tilde K_{ij}$ are  kernels
of type 1, so that, by \cite{folland}, Proposition 1.11 we have
\begin{equation}\label{27.2:2}
|\scal{f_\ell}{(d_c \mc K\phi)_\ell}_{L^2(\he n)}| \le  C\|f \|_{L^1(\he{n},E_0^{n+1})}\|\phi\|_{L^{Q/2}(\he{n},E_0^n)}.
\end{equation}
The same argument can be carried out for all the components of $f$, yielding
\begin{equation}\label{27.2:3}
|\scal{f}{d_c \mc K\phi}_{L^2(\he n, E_0^{n+1})}| \le  C\|f \|_{L^1(\he{n},E_0^{n+1})}\|\phi\|_{L^{Q/2}(\he{n},E_0^n)}.
\end{equation}
%
Consider now the second term in \eqref{representation 3}. We have
\begin{equation*}\begin{split}
\scal{ u}{(d_c\delta_c)^2\mc K\phi}_{L^2(\he n,E_0^n)} &=
\scal{d_c \delta_c  u}{d_c\delta_c\mc K\phi}_{L^2(\he n,E_0^n)} 
\\&
= \scal{d_c g}{d_c\delta_c\mc K\phi}_{L^2(\he n,E_0^n)} .
\end{split}\end{equation*}
We notice now that $d_c g$ is a $d_c$-closed form in $E_0^n$, and then we can repeat the arguments
leading to \eqref{27.2:5} for $f$ in the case $h=n-1$, obtaining
\begin{equation}\label{eq g}
\big|\scal{d_c g}{d_c\delta_c\mc K\phi}_{L^2(\he n,E_0^n)} \big| \le \|d_cg\|_{L^1(\he n,E_0^n)}\|\nabla_{\he{}}d_c \delta_c\mc K\phi\|_{L^Q(\he n,E_0^n)}
\end{equation}
As above, $\nabla_{\he{}}d_c \delta_c \mc K\phi$ can be expressed as a sum of terms with components of the form
$$
\phi_j\ast W^I\tilde K_{ij}, \quad\mbox{with $d(I)=3$,}
$$
since $\delta_c :E_0^n\to E_0^{n-1}$ is an operator of order 1 in the horizontal derivatives,
as well as $d_c :E_0^{n-1}\to E_0^n$.
By Theorem \ref{global solution}, iv) and Proposition \ref{kernel}, ii) $W^I\tilde K_{ij}$ are  kernels
of type 1, so that, by \cite{folland}, Proposition 1.11 we have
\begin{equation*}
\big|\scal{d_c g}{d_c\delta_c\mc K\phi}_{L^2(\he n,E_0^n)} \big|
 \le  C\|d_c g \|_{L^1(\he{n},E_0^n)}\|\phi\|_{L^{Q/2}(\he{n},E_0^n)}.
\end{equation*}

Combining this estimate with the one in \eqref{27.2:3}, we get eventually
$$
| \scal{u}{\phi}_{L^2(\he n, E_0^n)} | \le C\big(\|f\|_{L^1(\he n, E_0^{n+1})} + 
\|d_cg\|_{L^1(\he n, E_0^n)} \big) \|\phi\|_{L^{Q/2}(\he n, E_0^n)},
$$
and hence
$$
\|u\|_{L^{Q/(Q-2)}(\he{n}, E_0^n)}\le C\big(\|f\|_{L^1(\he n, E_0^{n+1})} + 
\|d_cg\|_{L^1(\he n, E_0^n)} \big).
$$
To achieve the proof of statement iv) of the theorem we have to consider separately
the cases $f=0$ and $g=0$. Suppose $h=n+1$ (i.e. $g=0$). The proof for $h=n$
(i.e. $f=0$) follows by Hodge duality. 
In the case $h=n+1$ identity \eqref{representation 3} read as
\begin{equation*}
\begin{split}
\scal{u}{\phi}_{L^2(\he n,E_0^{n+1})} &= \scal{ u}{\Delta_{\he{},n+1}\mc K \phi}_{L^2(\he n,E_0^{n+1})}
= \scal{ u}{ (\delta_c d_c)^2 \mc K\phi}_{L^2(\he n,E_0^{n+1})}
\\&
 =\scal{ f}{ d_c\delta_cd_c\mc K\phi}_{L^2(\he n,E_0^{n+2})}.
\end{split}
\end{equation*}
Since $d_cf=0$, by Theorem \ref{pierre} we can apply Theorem \ref{chanillo_van}, and we get
\begin{equation*}
\begin{split}
\big| \scal{u}{\phi}_{L^2(\he n,E_0^{n+1})}\big| & \le
C \|f\|_{L^1(\he n,E_0^{n+2})}\|\nabla_{\he{}}  d_c\delta_c d_c\mc K\phi\|_{L^Q(\he n,E_0^{n+2})}
\\&
\le  C \|f\|_{L^1(\he n,E_0^{n+2})}\|\phi\|_{L^Q(\he n,E_0^{n+1})},
\end{split}
\end{equation*}
by \cite{folland}, Proposition 1.9, since $\nabla_{\he{}} d_c\delta_c d_c \mc K$ is a kernel of type 0.
Then we can conclude by duality as of the proof of case iii), achieving the proof of statement iv)
of the theorem that now is completely proved.

\end{proof}

\section{Final remarks} \label{final}
The estimates in Theorem \ref{Hn} for $n$-forms and $(n+1)$-forms can be reformulated in the spirit of the
estimates proved in \cite{BFTr}. To state our result, we must recall preliminarily  few definitions of the function spaces we need for our results.

If $p,q\in [1,\infty]$, we define the space
$$
L^{p,q}(\he n) := L^p(\he n)\cap L^q(\he n)
$$
endowed with the norm
$$
\| u\|_{L^{p,q}(\he n)} :=  (\|u\|_{L^p(\he n)} ^2+ \|u\|_{L^q(\he n)}^2)^{1/2}.
$$
We have:
\begin{itemize}
\item $L^{p,q}(\he n)$ is a Banach space;
\item  $\mc D(\he n)$ is dense in $L^{p,q}(\he n)$.
\end{itemize}
Again if $p,q\in [1,\infty]$, we can endow the vector space $L^p(\he n) + L^q(\he n)$ with the norm
\begin{equation*}\begin{split}
\| u  \|_{L^p(\he n) + L^q(\he n)} : = \inf
\{  & ( \|u_1\|_{L^p(\he n)}^2 + \|u_2\|_{L^q(\he n)}^2)^{1/2};  \; 
\\&
u_1\in L^p(\he n), u_2 \in L^q(\he n), u = u_1+u_2\}.
\end{split}\end{equation*}
We stress that $L^p(\he n) + L^q(\he n)\subset L^1_{\mathrm{loc}}(\he n)$. Analogous spaces
of forms can be defined in the usual way.

The following characterization of $(L^{p,q}(\he n))^*$ can be proved by standard arguments
of functional analysis.

\begin{proposition}\label{duality} If $p,q\in (1,\infty)$
 and $p',q'$ are their conjugate exponents, then
 \begin{itemize}
 
\item[i)] if $u=u_1+u_2\in L^{p'}(\he n) + L^{q'}(\he n)$, with
$u_1\in L^{p'}(\he n)$ and $u_2\in L^{q'}(\he n)$, then the map
\begin{equation*}
\phi \to \int_{\he n} (u_1\phi + u_2\phi)\, dp \quad\mbox{for $\phi\in L^{p,q}(\he n)$}
\end{equation*}
belongs to $(L^{p,q}(\he n))^*$ and $\|u\|_{L^{p'}(\he n) + L^{q'}(\he n)}\ge \|F\|$;

\item[ii)] if $u\in L^{p'}(\he n) + L^{q'}(\he n)$, then there exist
$u_1\in L^{p'}(\he n)$ and $u_2\in L^{q'}(\he n)$ such that $u=u_1+u_2$
and $\|u\|_{L^{p'}(\he n) + L^{q'}\he n)} =  ( \|u_1\|_{L^p(\he n)}^2 + \|u_2\|_{L^q(\he n)}^2)^{1/2} $.
Moreover the functional 
\begin{equation*}
\phi \to F(\phi):= \int_{\he n} (u_1\phi + u_2\phi)\, dp\quad\mbox{for $\phi\in L^{p,q}(\he n)$}
\end{equation*}
belongs to $(L^{p,q}(\he n))^*$ and $\|F\|\approx \|u\|_{L^{p'}(\he n) + L^{q'}(\he n)} $.

\item[iii)] reciprocally, if $F\in (L^{p,q}(\he n))^*$,  then there exist 
 $u_1\in L^{p'}(\he n)$ and $u_2\in L^{q'}(\he n)$ such that
\begin{equation*}
F(\phi) = \int_{\he n} (u_1\phi + u_2\phi)\, dp\quad\mbox{for all $\phi\in L^{p,q}(\he n)$.}
\end{equation*}
If we set $u:=u_1+u_2\in L^{p'}(\he n) + L^{q'}(\he n)$,
then  $\|u\|_{L^{p'}(\he n) + L^{q'}\he n)}= \|F\|$.

\item[iv)] if $u_0\in L^1_{\mathrm{loc}}(\he n)$ satisfies
$$
\left| \int_{\he n} u_0\phi\; dp \right|  \le C \|\phi \|_{ L^{p',q'}(\he n) }
$$
for some $C>0$ and for all $\phi\in \mc D(\he n)$, then $u_0\in L^{p'}(\he n) + L^{q'}(\he n)$
and 
$$\|u_0\|_{L^{p}(\he n) + L^{q}(\he n)} \le C.$$
\end{itemize}
\end{proposition}

Using the function spaces defined above, we can reformulate Theorem \ref{Hn} in the critical cases
$h=n$ as follows. Since a similar formulations lacks in \cite{BF7}, we state and prove the theorem also of $n=1$.

\begin{theorem}\label{Hn bis} Denote by $(E_0^*,d_c)$ the Rumin's complex in $\he n$, $n\geq 1$.
Consider the system
$$
\left\{\begin{aligned}
 d_c u = f \\ \delta_c u = g&
\end{aligned}
\right.
$$
\indent If $n\geq 2$, then there exists $C>0$ such that  
 \begin{align*} 
&\|u\|_{L^{Q/(Q-2)}(\he{n}, E_0^n)+ L^{Q/(Q-1)}(\he{n}, E_0^n)}\le   C\big( \|f \|_{L^1(\he{n}, E_0^{n+1})} + \|g \|_{L^1(\he{n}, E_0^{n-1})}\big) ;
\\
&\|u\|_{L^{Q/(Q-2)}(\he{n}, E_0^{n+1} ) + L^{Q/(Q-1)}(\he{n}, E_0^{n+1})}\le   C\big(\|f \|_{L^1(\he{n}, E_0^{n+2})} + \|g \|_{L^1(\he{n}, E_0^{n})}  \big) ,
\end{align*}
for any  $u\in \mc D(\he n, E_0^n)$ and for any 
$u\in \mc D(\he n, E_0^{n+1})$, respectively.

If $n=1$,
 \begin{align*} 
&\|u\|_{L^{Q/(Q-2)}(\he{1}, E_0^1)+ L^{Q/(Q-1)}(\he{1}, E_0^1)}\le   C\big( \|f \|_{L^1(\he{1}, E_0^{2})} + \|g \|_{\mc H^1(\he{1}, E_0^{0})}\big) ;
\\
&\|u\|_{L^{Q/(Q-2)}(\he{1}, E_0^{2} ) + L^{Q/(Q-1)}(\he{1}, E_0^2)}\le   C\big(\|f \|_{\mc H^1(\he{1}, E_0^{3})} + \|g \|_{L^1(\he{1}, E_0^{1})}  \big) ,
\end{align*}
for any  $u\in \mc D(\he 1, E_0^1)$ and for any 
$u\in \mc D(\he 1, E_0^{2})$, respectively.

\end{theorem} 

\begin{proof} By Hodge duality we may
restrict ourselves to the case $h=n$. 

 If $u,\phi\in E_0^n$ are smooth compactly supported
forms, then we can write
\begin{equation}\label{representation 3bis}
\begin{split}
\scal{u}{\phi}_{L^2(\he n,E_0^n)} & = \scal{ u}{\Delta_{\he{},n}\mc K \phi}_{L^2(\he n,E_0^n)}
\\&
= \scal{ u}{(\delta_cd_c + (d_c\delta_c)^2)\mc K\phi}_{L^2(\he n,E_0^n)}
\\&
= \scal{ u}{\delta_c d_c\mc K\phi }_{L^2(\he n,E_0^n)} + \scal{ u}{(d_c\delta_c)^2)\mc K\phi}_{L^2(\he n,E_0^n)}
\\&
= \scal{d_c u}{d_c\mc K\phi }_{L^2(\he n,E_0^{n+1})} + \scal{ \delta_c u}{\delta_cd_c\delta_c\mc K\phi}_{L^2(\he n,E_0^{n-1})}
\\&
= \scal{f}{d_c\mc K\phi }_{L^2(\he n,E_0^{n+1})} + \scal{g}{\delta_cd_c\delta_c\mc K\phi}_{L^2(\he n,E_0^{n-1})}.
\end{split}
\end{equation}
The estimate of the first term of the last line of \eqref{representation 3bis} is already given in \eqref{27.2:3} and reads
\begin{equation}\label{27.2:3bis}
|\scal{f}{d_c \mc K\phi}_{L^2(\he n, E_0^{n+1})}| \le  C\|f \|_{L^1(\he{n},E_0^{n+1})}\|\phi\|_{L^{Q/2}(\he{n},E_0^n)},
\end{equation}
and, eventually,
\begin{equation}\label{27.2:6bis}\begin{split}
& |\scal{f}{d_c \mc K\phi}_{L^2(\he n, E_0^{n+1})}|
\\&
\hphantom{xxxx}
 \le  C\|f \|_{L^1(\he{n},E_0^{n+1})}
\left( \|\phi\|_{L^{Q/2}(\he{n},E_0^n)} +  \|\phi\|_{L^{Q}(\he{n},E_0^n)}\right).
\end{split}\end{equation}

Consider now the second term in the last line of \eqref{representation 3bis}. If $n\geq 2$, we have
\begin{equation}\label{last1}\begin{split}
 \scal{g}{\delta_cd_c\delta_c\mc K\phi}_{L^2(\he n,E_0^{n-1})} =
  \scal{\ast g}{\ast \delta_cd_c\delta_c\mc K\phi}_{L^2(\he n,E_0^{n+2})}
\end{split}\end{equation}
We notice now that $\ast g$ is a $d_c$-closed form in $E_0^{n+2}$. Then we can repeat the arguments
of the proof of Theorem \ref{Hn} and we get
\begin{equation}\label{eq g bis}\begin{split}
&\big|\scal{\ast g}{\ast \delta_c d_c\delta_c\mc K\phi}_{L^2(\he n,E_0^{n+2})} \big| 
\\&
\hphantom{xxxxx}
\le \|g\|_{L^1(\he n,E_0^{n-1})}\|\nabla_{\he{}}\delta_c d_c \delta_c\mc K\phi\|_{L^Q(\he n,E_0^n)}
\end{split}\end{equation}
As in the proof of Theorem \ref{Hn}, $\nabla_{\he{}}\delta_c d_c \delta_c \mc K\phi$ can be expressed as a sum of terms with components of the form
$$
\phi_j\ast W^I\tilde K_{ij}, \quad\mbox{with $d(I)=4$,}
$$
since $\delta_c :E_0^n\to E_0^{n-1}$ is an operator of order 1 in the horizontal derivatives,
as well as $d_c :E_0^{n-1}\to E_0^n$.
By Theorem \ref{global solution}, iv) and Proposition \ref{kernel}, ii) $W^I\tilde K_{ij}$ are  kernels
of type 0, so that, keeping in mind \eqref{last1}, by Proposition 1.9 we have
\begin{equation}\label{last}\begin{split}
&\big|\scal{ g}{\delta_c d_c\delta_c\mc K\phi}_{L^2(\he n,E_0^{n-1})} \big|
 \\&
\hphantom{xxxx}
 \le  C\| g \|_{L^1(\he{n},E_0^{n-1})}
\left( \|\phi\|_{L^{Q/2}(\he{n},E_0^n)} +  \|\phi\|_{L^{Q}(\he{n},E_0^n)}\right).
\end{split}\end{equation}

If $n=1$, we write instead
\begin{equation*}\begin{split}
 \scal{g}{\delta_cd_c\delta_c\mc K\phi}_{L^2(\he 1,E_0^{0})} =
  \scal{g\ast\tilde{K}}{\phi}_{L^2(\he 1,E_0^{0})},
\end{split}\end{equation*}
where $\tilde{K}$ is a kernel of type 1. By H\"older inequality and \cite{folland_stein}, Theorem 6.10,
\begin{equation*}\begin{split}
&|\scal{g\ast\tilde{K}}{\phi}_{L^2(\he 1,E_0^{0})}|
\\&\leq\|g\ast\tilde{K}\|_{L^{Q/(Q-1)}(\he 1,E_0^{0})}\|\phi\|_{L^{Q}(\he 1,E_0^{0})}\\
&\leq\|g\|_{\mc H^1(\he 1,E_0^{0})}\|\phi\|_{L^{Q}(\he 1,E_0^{0})}.
\end{split}
\end{equation*}
This yields
\begin{equation}\label{last34}\begin{split}
\scal{g}{\delta_cd_c\delta_c\mc K\phi}_{L^2(\he 1,E_0^{0})}
\leq\|g\|_{\mc H^1(\he 1,E_0^{0})}(\|\phi\|_{L^{Q}(\he 1,E_0^{0})}\|+\|\phi\|_{L^{Q/2}(\he 1,E_0^{0})}).
\end{split}\end{equation}

To conclude the proof, if $n>1$, combining \eqref{last} with \eqref{27.2:6bis} or \eqref{last34}, we get eventually
\begin{equation}\label{last2}\begin{split}
&\big|\scal{ g}{\delta_c d_c\delta_c\mc K\phi}_{L^2(\he n,E_0^{n-1})} \big|
 \le  C \left( \|f\|_{L^1(\he{n},E_0^{n+1})} + \| g \|_{L^1(\he{n},E_0^{n-1})}\right)
  \\&
\hphantom{xxxx}
\cdot
\left( \|\phi\|_{L^{Q/2}(\he{n},E_0^n)} +  \|\phi\|_{L^{Q}(\he{n},E_0^n)}\right).
\end{split}\end{equation}
If $n=1$ the same estimate holds with  $\| g \|_{L^1}$ replaced with $\| g \|_{\mc H^1}$.

Indeed,  if $n>1$ and we replace \eqref{27.2:6bis} and \eqref{last2} in \eqref{representation 3bis}, we obtain by duality
(Proposition \ref{duality} - iv)\;)
\begin{equation*}\begin{split}
& \|u\|_{L^{Q/(Q-2)}(\he{n}, E_0^n)+ L^{Q/(Q-1)}(\he{n}, E_0^n)}
\\&\hphantom{xxxxx}
\le   C\big( \|f \|_{L^1(\he{n}, E_0^{n+1})} + \|g \|_{L^1(\he{n}, E_0^{n-1})}\big).
\end{split}\end{equation*}
Again, $\| g \|_{L^1}$ must be replaced by $\| g \|_{\mc H^1}$ if $n=1$. This completes the proof of the theorem.

\end{proof}

\bibliographystyle{plain}

\bibliography{BFP_submitted}

\bigskip
\tiny{
\noindent
Annalisa Baldi and Bruno Franchi 
\par\noindent
Universit\`a di Bologna, Dipartimento
di Matematica\par\noindent Piazza di
Porta S.~Donato 5, 40126 Bologna, Italy.
\par\noindent
e-mail:
annalisa.baldi2@unibo.it, 
bruno.franchi@unibo.it.
}

\medskip

\tiny{
\noindent
Pierre Pansu 
\par\noindent Laboratoire de Math\'ematiques d'Orsay
\par\noindent UMR 8628 du CNRS
\par\noindent Universit\'e Paris-Sud
\par\noindent B\^atiment 425, Campus d'Orsay, 91405 Orsay, France.
\par\noindent
e-mail: pierre.pansu@math.u-psud.fr
}

\end{document}